\documentclass[12pt]{article}

\usepackage{paralist, amsmath, amsthm, amssymb, color, graphicx, hyperref}
\usepackage[margin=1in]{geometry}
\DeclareGraphicsExtensions{.jpg,.pdf,.eps}
\usepackage{amssymb}
\usepackage{setspace}
\usepackage{mathtools}
\usepackage{color}
\usepackage{cancel}
\usepackage{algorithm,algorithmic}

\usepackage{verbatim}
\usepackage[round]{natbib}

\newcommand{\overbar}[1]{\mkern 1.5mu\overline{\mkern-1.5mu#1\mkern-1.5mu}\mkern 1.5mu}

\usepackage{stmaryrd}

\usepackage{xfrac}

\usepackage{mathrsfs}
\usepackage{amscd}
\usepackage{dsfont}
\usepackage{tikz}
\usepackage{thm-restate}

\usepackage{svg}
\usepackage{bbding}

\usepackage{enumerate}

\theoremstyle{plain}
\newtheorem{thm}{Theorem}[section] 

\newtheorem{prop}[thm]{Proposition}

\newtheorem{lemma}[thm]{Lemma}

\theoremstyle{definition}
\newtheorem{defn}[thm]{Definition} 

\newtheorem{example}[thm]{Example}
\newtheorem{remark}[thm]{Remark}

\usepackage{listings}
\usepackage{color} 
\definecolor{mygreen}{RGB}{28,172,0} 
\definecolor{mylilas}{RGB}{170,55,241}
\definecolor{mygray}{gray}{0.95}

\usepackage{bbm}

\newcommand{\mscr}[1]{\mathscr{#1}}

\newcommand{\twid}[1]{\widetilde{#1}}

\newcommand{\ZZ}{\mathbb{Z}}
\newcommand{\RR}{\mathbb{R}}

\newcommand{\EE}{\mathbb{E}}

\newcommand{\al}{\alpha}

\newcommand{\ga}{\gamma}
\newcommand{\de}{\delta}
\newcommand{\ep}{\varepsilon}

\newcommand{\ta}{\theta}
\newcommand{\Ta}{\Theta}
\newcommand{\la}{\lambda}
\newcommand{\sa}{\sigma}

\newcommand{\De}{\Delta}

\usepackage{mathtools}

\renewcommand{\l}{\left}
\renewcommand{\r}{\right}

\newcommand{\defeq}{\vcentcolon=}

\DeclareMathOperator{\argmax}{argmax}

\newcommand{\iid}{\overset{\text{iid}}{\sim}}

\newcommand{\ol}{\overline}
\newcommand{\ul}{\underline}

\title{Simulation-based, Finite-sample Inference\\ for Privatized Data}

\author{Jordan Awan and Zhanyu Wang\\ Purdue University, Department of Statistics}
\date{}

\begin{document}
\setcounter{tocdepth}{2}

\maketitle


\begin{abstract}
Privacy protection methods, such as differentially private mechanisms, introduce noise into resulting statistics which often produces complex and intractable sampling distributions. In this paper, we propose a simulation-based ``repro sample'' approach to produce statistically valid confidence intervals and hypothesis tests, which builds on the work of \citet{xie2022repro}. We show that this methodology is applicable to a wide variety of private inference problems, appropriately accounts for biases introduced by privacy mechanisms (such as by clamping), and improves over other state-of-the-art inference methods such as the parametric bootstrap in terms of the coverage and type I error of the private inference. We also develop significant improvements and extensions for the repro sample methodology for general models (not necessarily related to privacy), including 1) modifying the procedure to ensure guaranteed coverage and type I errors, even accounting for Monte Carlo error, and 2) proposing efficient numerical algorithms to implement the confidence intervals and  $p$-values. 
\end{abstract}

\noindent%

\section{Introduction}
There is a growing interest in developing privacy-preserving methodologies, and differential privacy (DP) has arisen as the state-of-the-art in privacy protection \citep{dwork2006calibrating}. However, when formal privacy protection methods are applied to sensitive data, the published quantities may have a very different structure/distribution than the original data, which depends on the privatization algorithm. 
A major benefit of differential privacy is that the privacy mechanism -- the noise-adding distribution -- can be publicly communicated without compromising privacy. Thus, a statistician or data analyst can, in principle, incorporate the uncertainty due to the privacy mechanism into their statistical reasoning. Despite this potential, there is a need for general-purpose methods to conduct valid statistical inferences on privatized data. A major challenge is that, given a model $\ul x \sim f(\ul x\mid\theta)$ for the sensitive data and the privacy mechanism $s|\ul x \sim \eta(s\mid \ul x)$, the marginal likelihood function for $\theta$ based on $s$, 
\[L(\theta\mid s) = \int_{\ul x} f(\ul x\mid \theta) \eta(s\mid \ul x) \ d\ul x,\]
is often computationally intractable, requiring integration over all possible databases $\ul x$. Recently, there have arisen Markov chain Monte Carlo methods to sample from the posterior distribution $\pi(\theta\mid s)$ \citep{bernstein2018differentially,bernstein2019differentially,kulkarni2021differentially,ju2022data}, with \citet{ju2022data} giving a general-purpose technique for Bayesian inference on privatized data. However, there are not as many general-purpose frequentist methods. Due to the intractability of the marginal likelihood, most common approaches use approximations, such as the parametric bootstrap \citep{ferrando2022parametric} or limit theorems \citep{wang2018statistical}. 

Unfortunately, while these approximate techniques are often reliable in non-private settings, they often do not have finite sample guarantees and may give unacceptable accuracy when applied to privatized data. For example, while the noise for privacy is often asymptotically negligible compared to the statistical estimation rate, researchers have found that arguments based on convergence in distribution often result in unsatisfactory approximations \citep{wang2018statistical}, resulting in unreliable inference. While alternative asymptotic regimes have been proposed (e.g., \citealp{wang2018statistical}), they are limited to specific models, mechanisms, and test statistics. \citet{ferrando2022parametric} proposed using the parametric bootstrap on privatized data, but in order to prove the consistency of their method, they assumed that the data was not clamped, while clamping is often necessary in DP mechanisms. Thus, as we demonstrate in this paper, the parametric bootstrap often gives unacceptable accuracy on privatized data,  likely due to the bias in the privatized estimator.

While the marginal likelihood may not be easily evaluated, it is often the case that both $f(\ul x\mid \theta)$ and $\eta(s\mid \ul x)$ are relatively easy to sample from. This aspect makes simulation-based inference a good candidate for analyzing privatized data. Our simulation-based (also known as indirect inference) framework is inspired by the ``repro samples'' framework of \citet{xie2022repro}, and the general setup is similar to  other simulation-based works as well \citep{guerrier2019simulation}, and is widely applicable to statistical problems, beyond DP. 

We assume that rather than directly sampling from our model/mechanism, given a parameter $\theta$, we can separately sample a ``seed'' $u$ from a known distribution $P$ (not involving $\theta$), and apply a function that combines the seed and the parameter to produce the observed random variable: $s=G(\theta,u)$. Intuitively, $s=G(\theta,u)$ can be thought of as the equation that generated the observed $s$, where $u$ is the source of randomness. Since the true parameter is unknown, in simulation-based inference, we simulate seeds $\{u_i\}_{i=1}^R$ from $P$ and hold these fixed during our search over the parameter space for the parameter(s) $\hat{\theta}$ that make the values $G(\hat\theta,u_1),\ldots, G(\hat\theta,u_R)$ ``similar'' to our observed  $s$, where ideally a higher similarity measure will result in better estimates $\hat\theta$ of $\theta$. { Intuitively, a confidence set consists of all $\hat\theta$ which produces samples ``similar'' to $s$; likewise, for a null hypothesis $H_0: \theta\in \Theta_0$, if no $\hat\theta\in \Theta_0$ results in samples ``similar'' to $\theta$, then we reject $H_0$. We give a construction for a broad class of ``similarity'' measures that allow us to give finite-sample guarantees for these confidence sets and hypothesis tests.} 


{\bf \noindent Our Contributions } In this paper { we propose a general-purpose simulation-based inference framework based on repro samples \citep{xie2022repro}, which is motivated by the challenges in analyzing privatized data, and results in the following contributions:}

\begin{itemize}
\item We modify the repro samples methodology { using simulation-based inference techniques} to ensure guaranteed coverage and type I errors, even accounting for Monte Carlo errors, and present the repro methodology with a new notation to better communicate the applicability of the framework.
\item We propose computationally efficient algorithms for computing $p$-values and confidence intervals/sets, { which maintain the theoretical validity}. 
\item We use our methodology to tackle several DP inference problems, giving provable inference guarantees and improving over state-of-the-art methods (such as the parametric bootstrap.)
\item We demonstrate that our methodology can account for biases introduced by DP mechanisms, such as due to clamping or other non-linear transformations. 
\end{itemize}

We emphasize that while our methodology is motivated by problems in privacy, our framework is general-purpose and can be applied to non-DP settings as well. Our approach is applicable when one either has a test statistic, or only a low dimensional summary statistic but need not have an estimator for the true parameters, let alone an understanding of the sampling distribution.

{\bf \noindent Organization }The remainder of the paper is organized as follows: In Section \ref{s:background} we set the notation for the paper and review some necessary background on differential privacy. In Section \ref{s:CI}, we start by giving an abstract framework for generating confidence sets in Section \ref{s:conf:abstract} { and show how our framework is related to the repro sample framework of \citet{xie2022repro} in Section \ref{s:repro}. Then we use our abstract framework} to develop simulation-based confidence sets in Section \ref{s:conf:simulation} with guaranteed coverage, and develop a numerical algorithm for simulation-based confidence intervals in Section \ref{s:conf:numerical}. In Section \ref{s:pvalue}, we review how repro samples can be used to test hypotheses, and give a { simulation-based method } for valid $p$-values, as well as an efficient numerical algorithm. In Section \ref{s:simulations}, we conduct several simulation studies, applying our methodology to various DP inference problems. We conclude in Section \ref{s:discussion} with some discussion and directions for future work. In Section \ref{s:cost}, we specifically discuss the choice of the test statistic in the repro methodology, and quantify the cost of over-coverage, which is common when using { our simulation-based approach}. 
Proofs and technical lemmas, as well as simulation details and additional simulation results are deferred to the supplement. Code to replicate the experiments of this paper can be found at \url{https://github.com/Zhanyu-Wang/Simulation-based_Finite-sample_Inference_for_Privatized_Data}.

{\bf \noindent Related work } The area of indirect inference was proposed by \citet{gourieroux1993indirect}, which initially gave a simulation-based approach to bias correction and parameter estimation. However, the inferential method proposed by \citet{gourieroux1993indirect} is based on large sample theory, which is often unreliable in privacy problems \citep{wang2018statistical}. \citet{guerrier2019simulation,guerrier2020asymptotically,zhang2022flexible} developed a more theoretical basis for the indirect inference approach, focusing on bias-correction. \citet{xie2022repro} developed the repro sample methodology, { and both their abstract framework and Monte Carlo methods serve as} the basis for the simulation-based inference techniques proposed in this paper. \citet{wang2022finite} applied the repro sample methodology to develop model and coefficient inference in high-dimensional linear regression. 

In the area of statistical inference on privatized data, there are several notable works. For Bayesian inference on privatized data, \citet{ju2022data} proposed a general Metropolis-within-Gibbs algorithm, which correctly targets the posterior distribution for a wide variety of models and privacy mechanisms. In the frequentist setting, \citet{wang2018statistical} developed a custom asymptotic framework that ensures, under certain circumstances, that the approximations produced are at least as accurate as similar approximations for the non-private data. While lacking finite-sample guarantees, \citet{ferrando2022parametric} showed that the parametric bootstrap is a very flexible technique to perform inference on privatized data.  

There have also been several works that tailor their statistical inference to particular models and mechanisms. 
\citet{awan2018differentially} developed uniformly most powerful DP tests for Bernoulli data, and \citet{awan2020differentially} extended this work to produce optimal confidence intervals as well. \citet{drechsler2022nonparametric} developed private non-parametric confidence intervals for the median of a univariate random variable. \citet{karwa2018finite} developed private confidence intervals for the mean of normally distributed data. \citet{covington2021unbiased} developed a DP mechanism to generate confidence intervals for general parameter estimation, based on the CoinPress algorithm \citep{biswas2020coinpress} and the bag of little bootstraps \citep{kleiner2012big}. 

In contrast with many of the related works above, the goal of this paper is to produce valid finite-sample frequentist inference for a wide variety of models and mechanisms. In particular, we will not ask that the privacy mechanism be tailored for our model or for our particular statistical task. While \citet{ju2022data} offered a general solution in the Bayesian setting, there are essentially no prior frequentist techniques to derive valid finite-sample inference for general models and mechanisms. 

\section{Background}\label{s:background}
In this section, we review the necessary background and set the notation for the paper. 

We call a function $T:\mscr X\times \mscr X^n\rightarrow \mscr Y$ \emph{permutation-invariant} if for any permutation $\pi$ on $\{1,2,\ldots, n\}$, we have $T(x;(x_i)_{i=1}^n)=T(x;(x_{\pi(i)})_{i=1}^n)$. Note that if $X_1,\ldots, X_n$ are exchangeable random variables in $\mscr X$, and $T$ is a permutation-invariant function, then the sequence of random variables, $T(X_1;(X_i)_{i=1}^n), T(X_2;(X_i)_{i=1}^n),\ldots, T(X_n;(X_i)_{i=1}^n)$ is also exchangeable. 

For a real value $x$, we define $\lfloor x\rfloor$ to be the greatest integer less than or equal to  $x$, and $\lceil x \rceil$ to be the smallest integer greater than or equal to  $x$. We define the \emph{clamp} function $[\cdot]_a^b:\RR\rightarrow [a,b]$ as $[x]_a^b=\min\{\max\{x,a\},b\}$. Many differentially private mechanisms use clamping to ensure finite sensitivity; see Example \ref{ex:additive}.

In this paper, we will use $s$ to denote the observed ``sample.'' While in many statistical problems, a sample consists of i.i.d.\ data, in this paper, we allow $s$ to generally be the set of observed random variables from an experiment. In differential privacy, the observed values after privatization are often low-dimensional quantities; see Example \ref{ex:binomial} below. { When working with a list of samples, $s_1,\ldots, s_R$, we use $s_i^{(j)}$ to denote the $j^{th}$ entry of the vector $s_i$.}

\begin{example}[Bernoulli example]\label{ex:binomial}
One of the simplest models is independent Bernoulli data. Suppose that $x_i \iid \mathrm{Bern}(\theta)$, and $(x_1,x_2,\ldots,x_n)$ represents a confidential dataset. If we are interested in doing inference on $\theta$, the total $X=\sum_{i=1}^n x_i$ is a sufficient statistic, but for privacy purposes, we will instead base our inference solely on $\tilde X=X + N$, where $N$ is a random variable independent of $X$ chosen such that $\tilde X$ satisfies differential privacy \citep{vu2009differential,awan2018differentially}. In this paper, we consider our ``sample'' to be the observed value $s\defeq \tilde X$.
\end{example}

\subsection{Differential privacy}\label{s:dp}
Differential privacy, introduced by \citet{dwork2006calibrating}, is a probabilistic framework used to quantify the privacy loss of a \emph{mechanism} (randomized algorithm). In the big picture, differential privacy requires that for any two \emph{neighboring} databases -- differing in one person's data -- the resulting distributions of outputs are ``close''. There have been many variants of differential privacy proposed, which alter both the notion of ``neighboring'' as well as ``close''-ness. Typically, the neighboring relation is expressed in terms of an \emph{adjacency} metric on the space of input databases, and the closeness measure is formulated in terms of a divergence or constraint on hypothesis tests. 

If $\mscr D$ is the space of input datasets, a metric $d:\mscr D\times \mscr D\rightarrow \RR^{\geq 0}$ is an \emph{adjacency} metric if $d(D,D')\leq 1$ represents that $D$ and $D'$ differ by one individual. When $d(D,D')\leq 1$, we say that $D$ and $D'$ are \emph{neighboring} or \emph{adjacent} datasets. A \emph{mechanism} $M:\mscr D\rightarrow \mscr Y$ is a randomized algorithm; more formally, for each $D\in \mscr D$, $M(D)|D$ is a random variable taking values in $\mscr Y$.
\begin{defn}
[Differential privacy: \citet{dwork2006calibrating}]\label{def:dp}
Let $\ep\geq 0$, let $d$ be an adjacency metric on $\mscr D$, and let $M:\mscr D\rightarrow \mscr Y$ be a mechanism. We say that $M$ satisfies $\ep$-differential privacy ($\ep$-DP) if $P(M(D)\in S) \leq \exp(\ep) P(M(D')\in S)$, 
for all $d(D,D')\leq 1$ and all measurable sets $S\subset \mscr Y$. 
\end{defn}

In many problems, the space of datasets can be expressed as $\mscr D = \mscr X^n$, where $\mscr X$ represents the space of possible contributions from one individual and $n$ is fixed. In this case, it is common to take the adjacency metric to be Hamming distance, which counts the number of different entries between two databases. This setup is commonly referred to \emph{bounded DP}. 
The inference framework proposed in this paper is applicable to general database spaces and metrics, but the examples bounded DP unless otherwise noted.

Another common formulation of DP is Gaussian-DP, where the measure of ``closeness'' is expressed in terms of hypothesis tests.

\begin{defn}
[Gaussian DP: \citealp{dong2022gaussian}]
Let $\mu\geq 0$, let $d$ be an adjacency metric on $\mscr D$, and let $M:\mscr D\rightarrow \mscr Y$ be a mechanism. We say that $M$ satisfies $\mu$-Gaussian DP ($\mu$-GDP) if for any two adjacent databases $D$ and $D'$ satisfying $d(D,D')\leq 1$, the type II error of any hypothesis test on 
$H_0: Z\sim M(D)$ versus $H_1: Z\sim M(D')$, where $Z$ is the observed output of $M$, 
is no smaller than $\Phi(\Phi^{-1}(1-\alpha)-\mu),$ 
where $\alpha$ is the type I error, and $\Phi$ is the cumulative distribution function (cdf) of a standard normal random variable. 
\end{defn}

One can interpret GDP as follows: a mechanism satisfies $\mu$-GDP if testing $H_0: Z \sim M(D)$ versus $H_1: Z\sim M(D')$ for adjacent databases is at least as hard as testing $H_0: Z\sim N(0,1)$ versus $H_1: Z\sim N(\mu,1)$. 

Differential privacy satisfies basic properties such as composition and post-processing:
\begin{itemize}
    \item \emph{Composition}: if a mechanism $M_1:\mscr D\rightarrow \mscr Y$ satisfies $\ep_1$-DP ($\mu_1$-GDP) and mechanism $M_2:\mscr D \rightarrow \mscr Z$ satisfies $\ep_2$-DP ($\mu_2$-GDP), then the composed mechanism which releases both $M_1(D)$ and $M_2(D)$ satisfies $(\ep_1+\ep_2)$-DP ($\sqrt{\mu_1^2+\mu_2^2}$-GDP). 
\item \emph{Post-processing}: if a mechanism $M:\mscr D\rightarrow \mscr Y$ satisfies $\ep$-DP ($\mu$-GDP) and $f:\mscr Y\rightarrow \mscr Z$ is another mechanism, then $f\circ M:\mscr D\rightarrow \mscr Z$ satisfies $\ep$-DP ($\mu$-GDP). 
\end{itemize}


\begin{example}[Additive noise mechanism]\label{ex:additive}
Given a statistic $T:\mscr D\rightarrow \RR^d$, and a norm $\lVert \cdot \rVert$ on $\RR^d$, an \emph{additive noise mechanism} samples a random vector $N$ independent of the dataset, and releases $\twid T(D) =T(D)+\Delta N$, where $\Delta = \sup_{d(D,D)\leq 1}\lVert T(D)-T(D')\rVert$ is the \emph{sensitivity} of $T$. The \emph{Laplace mechanism}  uses the $\ell_1$ norm and samples $N_i\iid \mathrm{Laplace}(0,1/\ep)$ for $i=1,\ldots, d$; then $\twid T(D)$ satisfies $\ep$-DP. The \emph{Gaussian mechanism} uses  the $\ell_2$ norm and samples $N\sim N_d(0,(1/\mu^2)I)$ and  $\twid T(D)$ satisfies $\mu$-GDP. 

Given a real-valued statistic of the form $T(D) = \frac 1n\sum_{i=1}^n t(x_i)$, where $D=(x_1,\ldots, x_n)$, many DP mechanisms  will often first alter $T$ using a clamp function, to ensure finite sensitivity: $T'(D) = \frac 1n \sum_{i=1}^n [t(x_i)]_a^b$, and then apply an additive noise mechanism to $T'(D)$. However, clamping is a non-linear function, which results in a bias.
\end{example}

\section{Confidence intervals}\label{s:CI}

Confidence intervals/sets are a fundamental statistical tool, which concretely communicate the uncertainty of a parameter estimate. In Section \ref{s:conf:abstract}, we describe a general and abstract framework for constructing a confidence set, which encompasses both traditional methods as well as the simulation-based method explored in this paper. { In Section \ref{s:repro} we show that the repro sample methodology of \citet{xie2022repro} can be expressed as a special case of our framework, and discuss the limitations of the Monte Carlo methods of \citet{xie2022repro}.} In Section \ref{s:conf:simulation}, we show that, using simulation-based inference, we can construct valid confidence sets with guaranteed coverage, even accounting for the randomness in the simulation. In Section \ref{s:conf:numerical}, we give a concrete algorithm to { find the smallest confidence interval containing the confidence set derived in Section \ref{s:conf:simulation} as well as a grid-based confidence set.}



\subsection{Abstract confidence sets}\label{s:conf:abstract}
In this section, we set up a general and abstract framework for constructing a confidence set for an unknown parameter $\theta$ in Lemma \ref{lem:confidence}.

{ The abstraction in Lemma \ref{lem:confidence} is motivated by the framework in \citet{xie2022repro}, but we present our framework using a different notation and perspective; in Section \ref{s:repro} we show that the framework of \citet{xie2022repro} is a special case of Lemma \ref{lem:confidence}. }

\begin{restatable}{lem}{lemconfidence}\label{lem:confidence}
Let $\alpha \in (0,1)$ be given. Let $\theta^*\in \Theta$ be the unknown parameter, $s$ be the observed sample where $s\sim F_{\theta^*}$, and $\omega \sim Q$ be a random variable which is independent of $s$. 
For any fixed $\theta$, let $B_\alpha(\theta;s,\omega)$ be an event, which depends on $s$ and $\omega$, such that for all $\theta \in \Theta$, 
\begin{equation}\label{eq:predictionSet}
P_{ s\sim F_\theta,\omega\sim Q}(B_\alpha(\theta; s,\omega))\geq 1-\alpha.
\end{equation}
Then 
\begin{equation}\label{eq:CI}
\Gamma_\alpha(s,\omega)\defeq \{\theta\mid I(B_\alpha(\theta;s,\omega))=1\},
\end{equation}
is a $(1-\alpha)$-confidence set for $\theta^*$.  
If $\theta = (\theta_1, \ldots, \theta_k)$ { can be decomposed to $\beta$ and $\eta$ where $\beta=(\theta_{i_1}, \cdots, \theta_{i_m})$, $1\leq i_1 < \cdots < i_m \leq k$ and $\eta$ contains the rest $\theta_i$ not included in $\beta$, denoted by $\theta=(\beta,\eta)$}, 
then 
\begin{equation}\label{eq:project}
\Gamma^\beta_\alpha(s,\omega) = \{\beta\mid \exists \eta \text{ s.t. } I(B_\alpha((\beta,\eta);s,\omega))=1\},
\end{equation}
is a $(1-\alpha)$-confidence set for $\beta$. More generally, if $\theta = (\theta_1,\ldots, \theta_k)$, then $\Gamma_\alpha^{\theta_1}(s,\omega),\ldots, \Gamma_{\alpha}^{\theta_k}(s,\omega)$ are simultaneous $(1-\alpha)$-confidence sets for $\theta_1^*,\ldots, \theta_k^*$. 
\end{restatable}

\begin{figure}[t]
    \centering
    \includegraphics[width=1\linewidth]{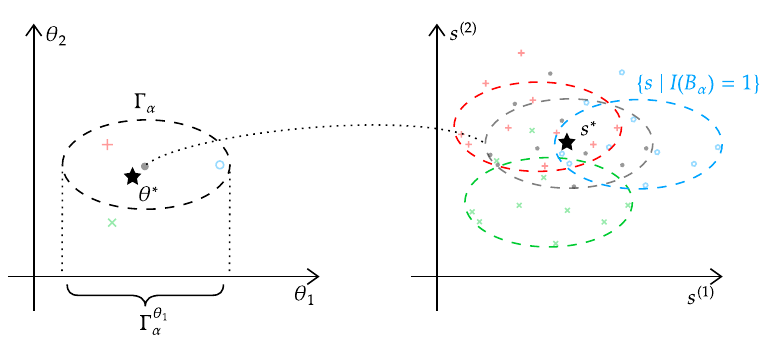}
    \caption{An illustration for $B_\alpha(\theta;s,\omega)$ and $\Gamma_\alpha(s,\omega)$ in Lemma \ref{lem:confidence}. The left subfigure is the space of parameter $\theta\in\mathbb{R}^2$, and the right subfigure is the space of $s\in\mathbb{R}^2$. The true parameter is $\theta^*$, and the observed sample is $s^*$, where both are denoted by $\bigstar$. For each $\theta$, we can construct $B_\alpha(\theta;s,\omega)$ (e.g., by simulation as in Theorem \ref{thm:construction}: the points in the right subfigure correspond to $\{s_i\}_{i=1}^R$), and we obtain $\Gamma_\alpha(s,\omega)$ by collecting the $\theta$ such that its corresponding $\{s|I(B_\alpha(\theta;s,\omega)) = 1\}$ contains $s^*$. For example, in the left subfigure, the parameters denoted by the red $+$ and blue $\circ$ are in $\Gamma_\alpha(s,\omega)$ while the one by green $\times$ is not. Note that $\Gamma_\alpha(s,\omega)$ is a valid confidence set with level $1-\alpha$ since for every $\theta \in \Theta$, we have $s \sim F_\theta$ being in $\{s|I(B_\alpha(\theta-;s,\omega)) = 1\}$ with probability $1-\alpha$, which means $P(\theta\in \Gamma_\alpha(s,\omega)) \geq 1-\alpha$. Furthermore, the confidence set $\Gamma_\alpha^{\theta_1}(s,\omega)$ is shown in the left subfigure, illustrating that it is the projection of $\Gamma_\alpha(s,\omega)$ onto the component $\theta_1$.}
    \label{fig:repro_illustration}
\end{figure}

One way of understanding the construction in Lemma \ref{lem:confidence} is as follows: for every candidate value of $\theta$, build an event $B_\alpha(\theta;s,\omega)$ which happens with probability $(1-\alpha)$ over the randomness of $s\sim F_\theta$ and $\omega\sim Q$. Then, given an observed pair $(s,\omega)$, our confidence set is simply the set of $\theta$ values for which $B_\alpha(\theta;s,\omega)$ did in fact happen. See Figure \ref{fig:repro_illustration} for an illustration. { Since $\omega$ is not part of our sample $s$, $\omega$ can be viewed as an extra source of randomness, which results in \emph{randomized} confidence sets. In Section \ref{s:conf:simulation}, we will use $\omega$ to account for the randomness in a simulation-based confidence interval.}

\begin{remark}
    We remark that the confidence sets $\Gamma_\alpha^{\theta_i}(s,\omega)$ can be viewed as a projection of $\Gamma_\alpha(s,\omega)$ onto the dimensions involving $\theta_i$, which is shown in the left subfigure in Figure \ref{fig:repro_illustration}. It is because each $\Gamma_\alpha^{\theta_i}(s,\omega)$ is implicitly constructed from a $\Gamma_\alpha(s,\omega)$ that we get the simultaneous coverage of $\Gamma_\alpha^{\theta_1}(s,\omega),\ldots, \Gamma_\alpha^{\theta_k}(s,\omega)$. As we will discuss in Section \ref{s:cost}, a downside of this construction is that the simultaneous nature of the confidence sets  can lead to over-coverage for any one of the $\Gamma_\alpha^{\theta_i}(s,\omega)$. 
\end{remark}

{ Most traditional constructions of confidence intervals fit in the framework of Lemma \ref{lem:confidence}, as illustrated in Example  \ref{ex:normalEasy}.} 

\begin{example}\label{ex:normalEasy}
We show how the classical method of deriving a confidence interval for the location parameter of the normal distribution fits into the framework of Lemma \ref{lem:confidence}. In this case, we will not need the random variable $\omega$. Let $s=(x_1,\ldots, x_n)$ where $x_i\iid N(\theta^*,\sigma^2_0)$, where only $\theta^*$ is unknown. Then $T_\theta(s)=\frac{\ol x-\theta}{\sigma_0/\sqrt{n}}\sim N(0,1)$ is a pivot where $\ol x = \frac{1}{n}\sum_{i=1}^n x_i$. We then take $C=[-z_{1-\alpha/2},z_{1-\alpha/2}]$, which is a $(1-\alpha)$ prediction interval for $T_{\theta^*}(s)$, where $\theta^*$ is the true parameter. Then $B_\alpha(\theta;s)=\{T_\theta(s)\in C\}$, and our confidence set is 
\begin{align*}
    \Gamma_\alpha(s)&=\{\theta| I(B_\alpha(\theta;s))=1\}\\
     &=\{\theta|T_\theta(s)\in [-z_{1-\alpha/2},z_{1-\alpha/2}]\}\\
     &=\left\{\theta\middle|-z_{\alpha/2}\leq \frac{\ol x-\theta}{\sigma_0/\sqrt n}\leq z_{\alpha/2}\right\}\\
     &=\left\{\theta\middle | \ol x-\sigma_0 z_{1-\alpha/2}/\sqrt n\leq \theta\leq \ol x+\sigma_0 z_{1-\alpha/2}/\sqrt n\right\}\\
     &=[\ol x-\sigma_0 z_{1-\alpha/2}/\sqrt n, \ol x+\sigma_0 z_{1-\alpha/2}/\sqrt n],
\end{align*}
which is the usual confidence interval for the mean.
\end{example}

Example \ref{ex:bern} below gives a DP example where the sampling distribution of $s$ can be evaluated numerically. However, in most DP problems, this is not the case. For the remainder of the paper, we will not assume that this sampling distribution is available, and in Section \ref{s:conf:simulation}, we will show how the auxiliary random variable $\omega$ allows us to use simulation-based inference to construct an event $B_\alpha$ satisfying \eqref{eq:predictionSet}.

\begin{example}[Bernoulli distribution: \citealp{awan2018differentially}]\label{ex:bern}
\citet{awan2018differentially} derived the uniformly most powerful DP hypothesis tests for Bernoulli data and showed that they could be expressed in terms of the test statistic $X+N$, where $X\sim \mathrm{Binomial}(n,p^*)$ and $N\sim \mathrm{Tulap}(0,b,q)$, where $X$ is based on the privatized data, and the parameters of the Truncated-Uniform-Laplace (Tulap) distribution\footnote{The Tulap distribution is closely related to the Staircase distributions \citep{geng2014optimal} and the Discrete Laplace distribution (used in the geometric mechanism) \citep{inusah2006discrete,ghosh2009universally}.} depend on the privacy parameters $\ep$ and $\de$. \citet{awan2018differentially} gave a closed-form expression for the cdf of $N$, and showed that the cdf of $X+N$ can be expressed as 
\[F_{X+N}(t) = \sum_{x=0}^n F_N(t-x)f_X(x),\]
where $f_X$ is the pmf of $X$, and $F_N$ is the cdf of $N$. This cdf is then used to derive $p$-values and confidence intervals  for $p^*$ in \citet{awan2018differentially,awan2020differentially}.
\end{example}

In the previous example, the convolution of $X$ and $N$ was tractable since there were only two variables, the pmf/cdf of $X$ and $N$ are easily evaluated, and $X$ took on only a finite number of values. \citet{awan2023canonical} showed that one could use the inversion of characteristic functions to numerically evaluate the cdf when convolving multiple known distributions, which they applied to the testing of two population proportions. In Section B.1 of the supplement, we show that for exponentially distributed data, we can apply this trick to account for clamping. However, it is not always tractable to evaluate the characteristic function of a clamped random variable.

\subsection{Repro sample confidence sets}\label{s:repro}

{ \citet{xie2022repro} propose a general framework to produce valid confidence sets using ``repro samples,'' which we show is a special case of Lemma \ref{lem:confidence}. In the case where a sampling distribution is unavailable, \citet{xie2022repro} propose an approximate Monte Carlo confidence set, which serves as the motivation for the method we propose in Section \ref{s:conf:simulation}. In this section, we review the key results from \citet{xie2022repro} which motivate our methods. }

We assume that the observed data $s$, drawn from a distribution parameterized by the true parameter $\theta^*\in \Theta$, can be expressed as 
\begin{equation}\label{eq:generating}
s\overset d=G(\theta^*,u),
\end{equation}
for a known measurable function $G$ and a random variable $u\sim P$, with known distribution $P$ (not depending on $\theta^*$). We will refer to equation \eqref{eq:generating} as a \emph{generating equation}, which can be interpreted as the equation that was used to generate the data \citep{hannig2009generalized}. We will sometimes refer to $u$ as the ``seed'' that produced the sample $s$. Generating equations like \eqref{eq:generating} are used in many areas of statistics, including fiducial statistics \citep{hannig2009generalized}, the reparametrization trick in variational autoencoders \citep{kingma2013auto}, and co-sufficient sampling \citep{engen1997stochastic,lindqvist2005monte}.
\begin{remark}
    { For many problems, there is a natural choice for $G$, such as a linear transformation for location-scale families and inverse-transform sampling for real-valued random variables. However, there is no general prescription for the generating equation, and models may have multiple generating equations. While the theory presented in this paper is equally valid for any generating equation, intuition suggests that it is preferable to choose $G$ that is smooth in both $\theta$ and $u$. We also note that for complex models, it may be easier to specify multiple ``mini'' generating equations for each step of a hierarchical model, which is often convenient when working with privatized data. See Example \ref{ex:bern} for an example.}
\end{remark}

{ \citet{xie2022repro} proposed the repro sample methodology as a theoretical framework to derive a confidence interval, given a generating equation $G(\theta,u)$, \emph{nuclear mapping} $T(u,\theta)$ (similar to a test statistic), and a $P$-probability $(1-\alpha)$ set $A_\alpha(\theta)$ on the range of $T$ (i.e., for all $\theta\in \Theta$, $P_{u\sim P}(T(u,\theta)\in A_\alpha(\theta))\geq 1-\alpha$):}
\begin{equation}\label{eq:repro}
    { \Gamma_\alpha(s) = \big\{\theta ~|~ \exists u^* \text{ s.t. } s = G(\theta,u^*),~ T(u^*,\theta)\in A_\alpha(\theta)\big\}.}
\end{equation}
{ \citet{xie2022repro} show that with this construction, $\Gamma_\alpha(s)$ is a valid confidence set for $\theta$. By taking $B_\alpha$ to be the event $\{\exists u^* \text{ s.t. } s = G(\theta,u^*), T(u^*,\theta)\in A_\alpha(\theta)\}$, we see that \eqref{eq:repro} is a special case of \eqref{eq:CI}.

When the sampling distribution of $T(u^*,\theta)$ for $u^*\sim P$ is unavailable, \citet{xie2022repro} propose a Monte Carlo algorithm, which uses $R$ simulated copies $u_1,\ldots, u_R\iid P$, and an empirical $(1-\alpha)$ set $\hat A_\alpha$ calculated from $G(u_1,\theta),\ldots, G(u_R,\theta)$ over a grid search of $\theta$. While they show that this procedure works well in practice, there is no theoretical guarantee that the resulting confidence sets are calibrated to the nominal level. The lack of calibration arises from two main issues:
\begin{enumerate}
    \item The empirical $\hat A_\alpha$ does not account for Monte Carlo errors, and it would be required that $R\rightarrow \infty$ in order to make the Monte Carlo errors negligible.
    \item The grid search over $\theta$ may miss disconnected or irregularly shaped regions. Even assuming that all connected regions are somewhat well-behaved, it would require that the density of the grid increases in order to identify all regions of the confidence set.
\end{enumerate}

Besides the above issues, the grid search method is computationally expensive, especially if one is only interested in producing a confidence interval for a single parameter. In the following subsection, we address these limitations with an improvement to the repro methodology using simulation-based inference techniques. }


\subsection{Simulation-based confidence { sets}}\label{s:conf:simulation}
Similar to { the Monte Carlo method of } \citet{xie2022repro}, we will use simulation techniques to build  the event $B_\alpha$ in Section \ref{s:conf:abstract}.  Unlike \citet{xie2022repro}, our result offers guaranteed coverage, even accounting for the Monte Carlo sampling in the procedure. { Going forward, we will refer to our simulation-based method as ``repro sample'' or ``repro'' and will reference \citet{xie2022repro} when discussing the original repro sample methodology.}


Using the same setup as in Section \ref{s:repro}, our procedure draws $R$ i.i.d. copies $u_1,\ldots,u_R\iid P$, and for each $\theta\in \Theta$, we consider $s_i(\theta)=G(\theta,u_i)$ for $i=1,\ldots, R$, which we call \emph{repro samples}. At the true parameter $\theta^*$, these Monte Carlo samples $\{s_i(\theta^*)\}_{i=1}^R$ have the same distribution as $s$. Then, considering the sequence $s,s_1(\theta),\ldots, s_R(\theta)$, which consists of the repro samples and the observed sample, we construct the event $B_\alpha(\theta;s,(u_i)_{i=1}^R)$ such that for every $\theta \in \Theta$, $P_{s\sim F_\theta,u_i\sim P}(B_\alpha(\theta;s,(u_i)_{i=1}^R))\geq 1-\alpha$. We then plug in $B_\alpha$ into either \eqref{eq:CI} or \eqref{eq:project}  to form a valid $(1-\alpha)$-confidence set. We offer a general and constructive method to produce a valid $B_\alpha(\theta;s,(u_i)_{i=1}^R)$ in Theorem \ref{thm:construction}.

\begin{restatable}{thm}{thmconstruction}\label{thm:construction}
Let $\alpha \in (0,1)$. Let $s\in \mscr X$ be the observed sample, where $s=G(\theta^*,u)$, $u\sim P$, and $\theta^*$ is the true parameter. Let $u_1,\ldots, u_R\iid P$, and set $s_i(\theta)=G(\theta,u_i)$ for $i=1,\ldots, R$. 
Let $T_\theta:\mscr X\times \mscr X^{R+1}\rightarrow \RR$ be a permutation-invariant function, which serves as a test statistic. 
Call $T_{\text{obs}}(\theta) = T_\theta\big(s;s,s_1(\theta),\ldots, s_R(\theta)\big)$, and let $T_{(1)}(\theta),\ldots, T_{(R+1)}(\theta)$ be the order statistics of $T_\theta\big(s;s,s_1(\theta),\ldots, s_R(\theta)\big), T_\theta\big(s_1(\theta);s,s_1(\theta),\ldots, s_R(\theta)\big),$\ $\ldots,T_\theta\big(s_R(\theta);s,s_1(\theta),\ldots,s_R(\theta)\big)$. 
Then for any $b-a\geq \lceil (1-\alpha)(R+1)\rceil-1$, the event $B_\alpha(\theta;s,(u_i)_{i=1}^R)=\{T_{\text{obs}}(\theta)
 \in [ T_{(a)}(\theta),T_{(b)}(\theta)]\}$ satisfies \eqref{eq:predictionSet}. 
Thus, using $B_\alpha$  in either \eqref{eq:CI} or \eqref{eq:project} of Lemma \ref{lem:confidence}, where $\omega = (u_1,\ldots, u_R)$, gives a $(1-\alpha)$-confidence set. 
\end{restatable}

Theorem \ref{thm:construction} gives a straightforward construction for sets satisfying \eqref{eq:predictionSet}, which, as Lemma \ref{lem:confidence} points out, lead to valid confidence sets. A key difference between these results and those of \citet{xie2022repro} is that our $B_\alpha$ sets exactly satisfy \eqref{eq:predictionSet}, even including the Monte Carlo errors. The key insights for Theorem \ref{thm:construction} are 1) the need for a permutation-invariant statistic, 2) the importance of including $s$ on the right hand side of $T$ to ensure exchangeability, and 3) the construction of prediction sets based on order statistics, using reasoning similar to that in conformal prediction \citep{vovk2005algorithmic}. See Lemma A.1 in the supplement for details. 

\begin{example}\label{ex:ab}
We list some examples of permutation-invariant functions, and show how $a$ and $b$ can be chosen for one-sided or two-sided criteria:
\begin{enumerate}
\item If $s\in \RR$, then setting $T_\theta(s;(s_i(\theta))_{i=1}^R)=s$  is trivially permutation-invariant. In this case, both large and small values of $s$ indicate that it is unusual, so we may set $a$ and $(R+1)-b$ to be approximately equal:  $a=\lfloor(\alpha/2)(R+1)\rfloor$ and $b=a+\lceil (1-\alpha)(R+1)\rceil -1$. 
\item More generally, if $T_\theta:\mscr X\rightarrow \RR$ is a test statistic, which may also depend on the parameter $\theta$, then this also satisfies the assumptions of Theorem \ref{thm:construction}. If large values indicate that it is unusual, we can set $a=1$ and $b=\lceil (1-\alpha)(R+1)\rceil$. On the other hand, if small values of $T$ indicate that it is unusual, then we set $a=\lfloor\alpha (R+1)\rfloor+1$ and $b=R+1$.
\item In general, most statistical depth functions are permutation-invariant, such as Mahalanobis depth, simplicial depth, and Tukey/Halfspace depth. Typically, depth statistics are designed such that lower depth corresponds to unusual points. For example, Mahalanobis depth is defined as $T(x; X) = [1+(x-\mu_X)' \Sigma_X^{-1} (x-\mu_X)]^{-1}$ where $\mu_X$ and $\Sigma_X$ are the sample mean and covariance of $X$. To use this depth as the $T$ in Theorem \ref{thm:construction}, we let $X=(s,s_1(\theta),\ldots,s_R(\theta))$ and $x$ be one of $s, s_1(\theta), \ldots,s_R(\theta)$, and set  $a=\lfloor\alpha (R+1)\rfloor+1$ and $b=R+1$. 
\end{enumerate}
\end{example}

\begin{remark}
We remark that depth statistics were also recommended by \citet{xie2022repro} for use in their Monte Carlo repro sampling framework, but their approach did not calibrate the confidence sets as we do in Theorem \ref{thm:construction}. 
\end{remark}

We end this section with two relatively simple examples, showing how one can apply our simulation-based framework to private inference problems. See Section \ref{s:simulations} for additional examples.

{ \begin{example}\label{ex:bern}

Suppose that $\ul x = (x_1,\ldots, x_n)$, $x_i \iid \mathrm{Bern}(\theta^*)$, and we observe the privatized statistic $s =  \sum_{i=1}^n x_i+N$, for some noise distribution $N$. Following \citet{awan2020differentially}, we use the $\mathrm{Tulap}(0,\exp(-1),0)$ distribution to achieve $1$-DP. According to \citet{awan2020differentially}, such a Tulap random variable can be sampled as $N\overset d= G_1-G_2+U$, where $G_1,G_2\iid \mathrm{Geom}(1-\exp(-1))$ and $U\sim \mathrm{Unif}(-1/2,1/2)$ are independent.

 In this setting, we can set $P_x\defeq \mathrm{Unif}(0,1)$ and sample $u^x_i\sim P_x$ for $i=1,\ldots, n$. Then using $G_x(\theta,u) = I(u\leq\theta)$, we have $G_x(\theta,u^x_i)\iid \mathrm{Bern}(\theta)$. For the privacy mechanism, we set $P_{\mathrm{DP}}=\mathrm{Tulap}(0,\exp(-1),0)$ and sample $u_{\mathrm{DP}}\sim P_{\mathrm{DP}}$. Using the function $G_s(\ul x,u_{\mathrm{DP}})=\sum_{i=1}^n x_i+ u_{\mathrm{DP}}$, we have that $s|\ul x \overset d =G_s(\ul x,u_{\mathrm{DP}})$. Then, the full generating equation is,
 $G(\theta,u) =  \sum_{i=1}^nI(u^x_i\leq \theta)+u_{\mathrm{DP}}$, where $u = ((u^x_i)_{i=1}^n,u_{\mathrm{DP}})$.

In this case, the exact sampling distribution of $s$ can be derived, which \citet{awan2020differentially} use to construct an asymptotically unbiased DP confidence interval of $\theta^*$ based on $s$, which has exact type I error and is near-optimal. Table \ref{tab:binomial_awan} is the result of a small simulation study showing that our simulation-based methodology is able to obtain nearly identical performance as \citet{awan2020differentially}, and does not require an analytical sampling distribution. We set $R=200$ and use Mahalanobis depth as $T$. 

\begin{table}[t]
    \centering
    \begin{tabular}{l|ll}
    & Empirical Coverage & Average Width \\\hline
Repro Sample & { 0.949 (0.007) } & { 0.1657 (0.0005)}\\\hline
\citep{awan2020differentially} & { 0.947 (0.007)} & { 0.1632 (0.0004)}
    \end{tabular}
    \caption{Nominal 95\% confidence intervals for the $\theta^*$ of Bernoulli data. True parameters are $\theta^*=0.2$. Sample $x_1,\ldots, x_{100} \iid \mathrm{Bern}(\theta^*)$. The results are under $1$-DP and computed over 1000 replicates.}
    \label{tab:binomial_awan}
\end{table}
\end{example}}

\begin{figure}[t]
    \centering
    \includegraphics[width=.24\linewidth]{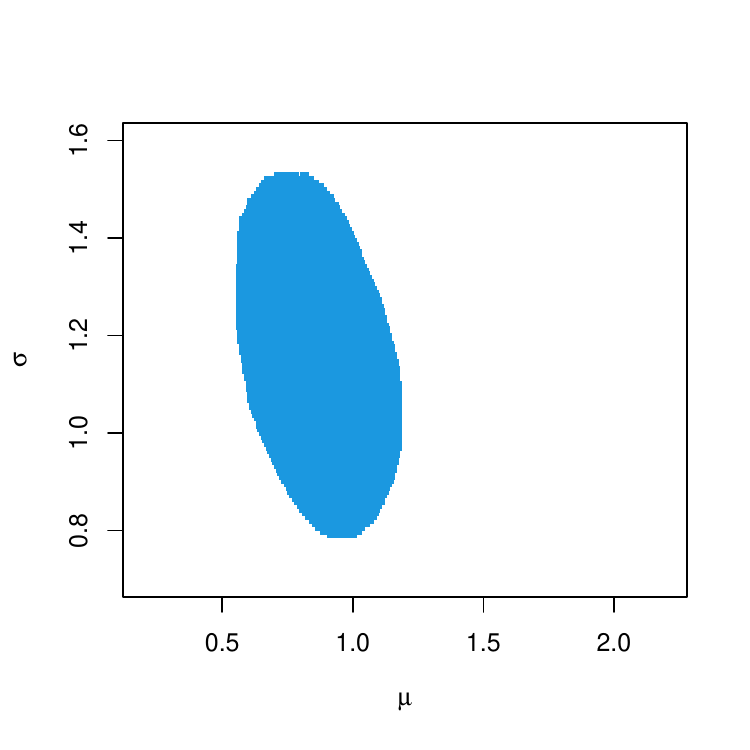}
    \includegraphics[width=.24\linewidth]{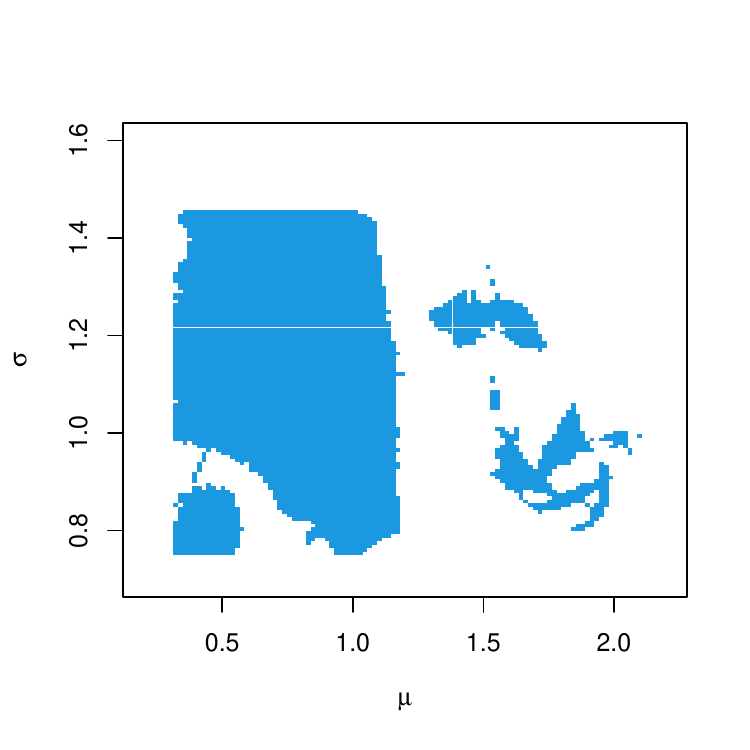}
    \includegraphics[width=.24\linewidth]{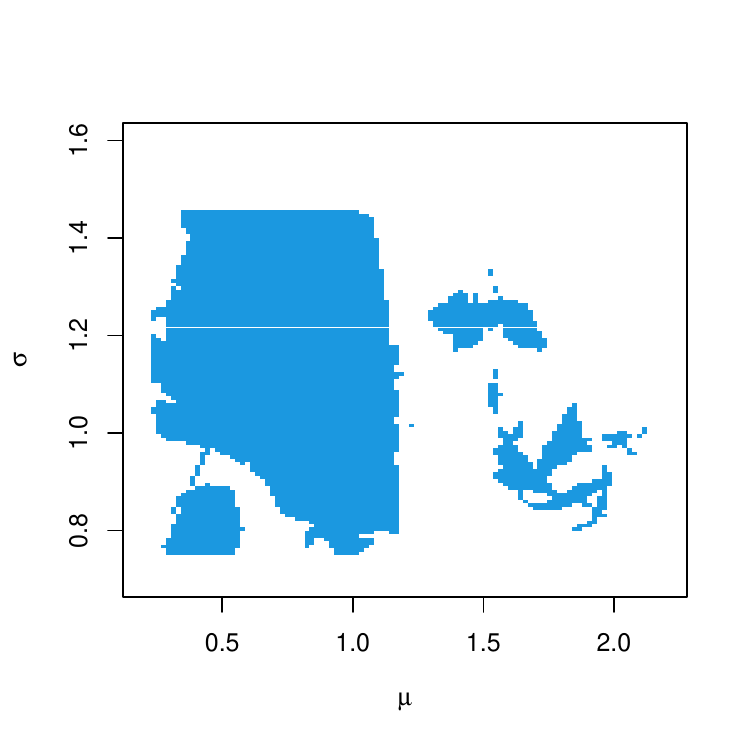}
    \includegraphics[width=.24\linewidth]{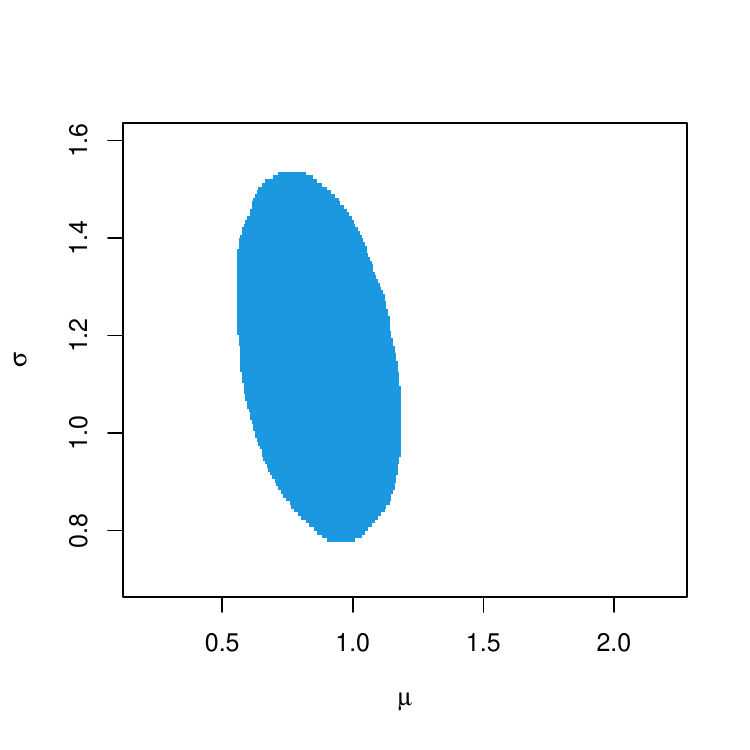}

    \caption{$95\%$ confidence set for location-scale normal (Example \ref{ex:normal1} and Section \ref{s:normal}), based on $s = (1, 0.75)$, generated using $n=100$, $\theta^*=(\mu^*,\sigma^*)=(1,1)$, $\epsilon=1$, $U=3$, $L=0$, and $R=200$ repro samples. From left to right: Mahalanobis depth (area 0.35), Halfspace/Tukey depth (area 0.61), Simplicial depth (area 0.63), Spatial depth (area 0.36).}
    \label{fig:NormSet}
\end{figure}

\begin{example}
[Comparison of depth statistics] \label{ex:normal1}
{ For privatized data, it can be challenging to develop a customized test statistic with good properties (such as being a pivot). Because of this, we explore various depth statistics which can be used as a default test statistic. In this example, we compare the confidence sets of different depth statistics in a location-scale normal problem, where we are given DP summary statistics which are the noisy mean and noisy variance of clamped data. See Section \ref{s:normal} for details on the model and privacy mechanism.}


For this problem, we have a two-dimensional privatized summary, one for the mean and one for the variance, and we compare different options of statistical depth for use in Theorem \ref{thm:construction}, see Figure \ref{fig:NormSet} { (using Algorithm \ref{alg:ConfGrid} in Section \ref{s:conf:numerical})}. { We see that the halfspace and simplicial depth statistics result in disconnected confidence regions, whereas Mahalanobis and spatial depth give better behaved confidence regions. Between Mahalanobis and spatial depth, we prefer Mahalanobis depth as it is more computationally efficient, has a simple formula, and is connected to other indirect inference techniques \citep{gourieroux1993indirect}. Note that while Mahalanobis depth results in a connected confidence region in this example, this need not generally be the case; in Section \ref{s:conf:numerical} we propose an algorithm which can account for disconnected regions in the confidence set. }
\end{example}

\subsection{Numerical algorithm for confidence intervals}\label{s:conf:numerical}

\begin{algorithm}[t]

\caption{{\texttt{accept}($\Theta_0\subset \Theta$, $\alpha$, $(u_i)_{i=1}^R$, $G$, $s$, $T$)}}
\scriptsize
INPUT: subset $\Theta_0\subset \Theta$ to search over, $\alpha\in [1/(R+1),1)$, seeds $(u_i)_{i=1}^R$, observed statistic $s\in \RR^d$, exchangeable statistic $T_\theta:\RR^d\times \RR^{d(R+1)}\rightarrow [0,1]$, defined for all $\theta\in \Theta$ (low values are interpreted as unusual), and $G(\theta,u)$ is the generating equation for $s$.

\begin{algorithmic}[1]
  \setlength\itemsep{0em}
  \STATE For a given $\theta$, denote $s_i(\theta) = G(\theta,u_i)$.
  \STATE Denote $T_i(\theta) = T_\theta(s_i(\theta);s,s_1(\theta),\ldots, s_R(\theta))$
  \STATE Denote $T_{\text{obs}}(\theta) = T_\theta(s;s,s_1(\theta),\ldots, s_R(\theta))$
  \STATE { Call $M = \sup_{\theta\in \Theta_0} [\#\{T_i(\theta)\leq T_{\text{obs}}(\theta)\}+1+T_{\text{obs}}(\theta)]$ }
  \IF{$M\geq \lfloor \alpha(R+1)\rfloor+1$}
  \STATE Return $\texttt{TRUE}$
  \ELSE
  \STATE Return $\texttt{FALSE}$
  \ENDIF
\end{algorithmic}

\label{alg:accept}
\end{algorithm}

\begin{algorithm}[t]
\caption{{ \texttt{confidenceInterval}($\alpha$, $\Theta = B\times H$, $(u_i)_{i=1}^R$, $G$, $s$, $T$, $\mathrm{tol}$)}}
\scriptsize
INPUT: $\alpha\in [1/(R+1),1)$, decomposition $\Theta=B\times H$ of $\theta = (\beta,\eta)$ such that $B\subset \RR$, seeds $(u_i)_{i=1}^R$, observed statistic $s\in \RR^d$, exchangeable statistic $T_\theta:\RR^d\times \RR^{d(R+1)}\rightarrow [0,1]$, defined for all $\theta\in \Theta$ (low values are interpreted as unusual), $G(\theta,u)$ is the generating equation for $s$, and $\mathrm{tol}>0$ is the numerical tolerance.

\begin{algorithmic}[1]
  \setlength\itemsep{0em}
  \IF{There exists $\hat\beta_{\text{init}}\in B$ such that $\texttt{accept}(\hat\beta_{\text{init}},\alpha,\Theta=B\times H,(u_i)_{i=1}^R,G,s,T)=\texttt{TRUE}$}

\IF{There exists $\ul{\beta_L}\in B$ such that $\texttt{accept}([\inf B,\ul{\beta_L}]\times H,\alpha,(u_i)_{i=1}^R,G,s,T)=\texttt{FALSE}$}
\STATE Perform a bisection method search as follows:
\STATE Set $\ol{\beta_L} = \hat \beta_{\text{init}}$
\WHILE{$\ol{\beta_L} - \ul{\beta_L}>\mathrm{tol}$}
\STATE Set $\beta_L^* = (\ul{\beta_L}+\ol{\beta_L})/2$
\IF{$\texttt{accept}([\ul{\beta_L},\beta_L^*]\times H,\alpha,(u_i)_{i=1}^R,G,s,T)=\texttt{TRUE}$}
\STATE Set $\ol{\beta_L} = \beta_L^*$
\ELSE
\STATE Set $\ul {\beta_L} = \beta_L^*$
\ENDIF
\ENDWHILE

\ELSE
\STATE $\ul {\beta_L}=\inf B$
\ENDIF
\IF{There exists $\ol{\beta_U}\in B$ such that $\texttt{accept}([\ol {\beta_U},\sup B],\alpha,(u_i)_{i=1}^R,G,s,T)=\texttt{FALSE}$}
\STATE Perform a bisection method search as follows:
\STATE Set $\ul{\beta_U} = \hat\beta_{\text{init}}$
\WHILE{$\ol{\beta_U} - \ul{\beta_U}>\mathrm{tol}$}
\STATE Set $\beta^*_U = (\ul {\beta_U }+ \ol {\beta_U})/2$
\IF{$\texttt{accept}([\beta_U^*,\ol{\beta_U}]\times H,\alpha,(u_i)_{i=1}^R,G,s,T)=\texttt{TRUE}$}
\STATE Set $\ul{\beta_U} = \beta_U^*$
\ELSE
\STATE Set $\ol{\beta_U} = \beta_U^*$
\ENDIF
\ENDWHILE
    \ELSE
\STATE $\ul{\beta_U} = \sup B$
\ENDIF
    \ELSE
    \STATE OUTPUT: $\emptyset$
    \ENDIF
\end{algorithmic}
OUTPUT: $[\ul {\beta_L},\ol{\beta_U}]$ 
\label{alg:CI}
\end{algorithm}

While Theorem \ref{thm:construction} gives a construction for confidence sets, which even accounts for Monte Carlo errors, finding an explicit description for the confidence set { is another challenge. 
In this section, we propose numerical algorithms that enable valid confidence sets which can be easily implemented. 
We first develop Algorithm \ref{alg:accept} which is determines whether a subset of the parameter space has intersection with the confidence set of Theorem \ref{thm:construction}, and use this to develop Algorithm \ref{alg:CI} which gives a valid confidence interval for a parameter of interest. We also propose Algorithm \ref{alg:ConfGrid} which is a valid grid-based multi-dimensional confidence set.}

\begin{restatable}{prop}{propalg}\label{prop:alg}
Let $\Theta=B\times H$ be a decomposition of a parameter space such that $B$ is a connected subset of $\RR$. Let $T$ be a permutation-invariant statistic taking values in $[0,1]$,  where small values give evidence that a sample is ``unusual'' relative to the others. { Let $\Gamma_\alpha^\beta(s,\omega)$  be the confidence set for $\beta\in B$ from \eqref{eq:project} based on $T$ as described in Theorem \ref{thm:construction}, using $a=\lfloor \alpha(R+1)\rfloor+1$ and $b=R+1$. Then, the output of Algorithm \ref{alg:CI} is a $(1-\alpha)$ confidence interval containing $\Gamma_\alpha^\beta(s,\omega)$, whose width is at most $2\mathrm{tol}$ units larger than the smallest interval containing $\Gamma_\alpha^\beta(s,\omega)$.}
\end{restatable}

\begin{example}[Example confidence sets]\label{ex:illustration}
    { In Figure \ref{fig:illustration}, we give illustrations on how Algorithms \ref{alg:CI} and \ref{alg:ConfGrid} operate. In part a) we run Algorithm \ref{alg:CI}, where at each point in time, regions are colored yellow to indicate that we are currently unsure whether a region belongs in the confidence interval or not, red if the region is confirmed to be outside of the confidence interval, and blue if the region is confirmed to belong to our final confidence interval. In the first line, we plot $\beta_{\mathrm{init}}$, which we know lies in the confidence set, and in the second line, we identify $\ul {\beta_L}$ and $\ol{\beta_U}$, which form initial lower and upper bounds for the confidence set. In line 3 we find that the interval $[\ul {\beta_L},\beta^*_L]$ does not intersect the confidence set and is marked red, and the interval $[\beta^*_L,\beta_{\mathrm{init}}]$ is marked yellow; we set $\ul {\beta_L}=\beta^*_L$ for the next iteration . We also confirm that the interval $[\beta^*_U,\ol{\beta_U}]$ does  intersect the confidence set and so is marked yellow, while this information implies that $[\beta_{\mathrm{init}},\beta^*_U]$ should be marked blue; we set $\ul {\beta_U}=\beta^*_U$ for the next iteration. Line 4 implements another iteration of the algorithm, and after this we output the resulting $ [\ul{\beta_L},\ol{\beta_U}]$.

    Part b) of Figure \ref{fig:illustration} shows the grid search confidence set of Algorithm \ref{alg:ConfGrid}, where the green region is the original confidence set $\Gamma_\alpha$ by Theorem \ref{thm:construction}, the blue regions are the output of Algorithm \ref{alg:ConfGrid}, which are confirmed to have intersection with $\Gamma_\alpha$, and the red regions are confirmed to have no intersection with $\Gamma_\alpha$. }
\end{example}

\begin{figure}[t]
    \centering
    \includegraphics[width=0.6\linewidth]{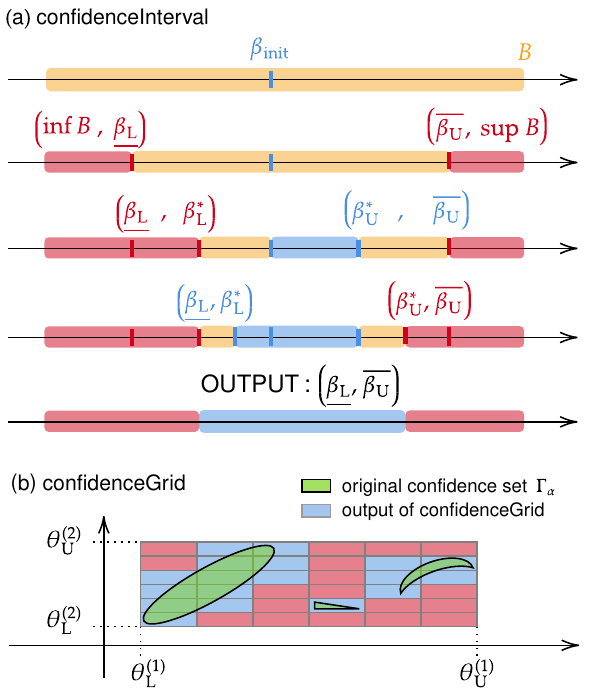}

    \caption{ (a) Illustration of Algorithm \ref{alg:CI}. (b) Illustration of Algorithm \ref{alg:ConfGrid}. See Example \ref{ex:illustration} for details. }
    \label{fig:illustration}
\end{figure}

\begin{remark}
    Step 1 of Algorithm \ref{alg:CI} requires us to find  $\hat\beta_{\text{init}}\in \Gamma_\alpha^\beta(s,u)$. Assuming that $\hat\beta_{\text{init}}=\argmax_{\beta} \{\sup_{\eta}\{\#\{T_{\text{obs}}((\beta,\eta))\geq T_{(i)}((\beta,\eta))\}\}\}$ is well defined and $\Gamma_\alpha^\beta(s,u)$ is non-empty, we see that $\hat\beta_{\text{init}}\in \Gamma_\alpha^\beta(s,u)$. Note that since we only need any point in $\Gamma_\alpha^\beta(s,u)$, we can also use { any other estimator, provided it is confirmed to belong to $\Gamma^\beta_\alpha(s,u)$}.  

    { To find the starting $\ul {\beta_L}$ (and similarly for $\ol {\beta_U}$) in Algorithm \ref{alg:CI}, one can use $\ul{\beta_L} = \inf B$ in the case that $\inf B>-\infty$. Otherwise, we can use a ``logarithmic bisection method'' starting at $\ul {\beta_L}=\min\{-1,\sup B\}$ and doubling $\ul {\beta_L}$ until we obtain a value that meets our requirements (if $\ul {\beta_L}$ diverges to $-\infty$, we can set $\ul {\beta_L}=-\infty$ in line 14). }
\end{remark}

\begin{remark}
In Algorithm \ref{alg:accept}, it is equivalent to set $M= \sup_{\theta\in \theta_0}\#\{T_i(\theta)\leq T_{\text{obs}}(\theta)\}+1$, as noted in the proof of Proposition \ref{prop:alg}. However, it is numerically challenging to maximize an integer-valued objective function. By altering the objective with the value $T_{\text{obs}}(\theta)$, we introduce a continuous component where typically, by increasing $T_{\text{obs}}(\theta)$, we will also increase the rank. This altered objective is more easily optimized by standard optimization software, such as R's L-BFGS-B method in \texttt{optim}. 
\end{remark}

\begin{algorithm}[t]
\caption{{ \texttt{confidenceGrid}($\alpha$, $\Theta$, $(u_i)_{i=1}^R$, $G$, $s$, $T$, $r$)}}
\scriptsize
INPUT: $\alpha\in [1/(R+1),1)$, $\Theta\ni \theta$, seeds $(u_i)_{i=1}^R$, observed statistic $s\in \RR^d$, exchangeable statistic $T_\theta:\RR^d\times \RR^{d(R+1)}\rightarrow [0,1]$, defined for all $\theta\in \Theta$ (low values are interpreted as unusual), $G(\theta,u)$ is the generating equation for $s$, and $r$ is the resolution of the grid.

\begin{algorithmic}[1]
  \setlength\itemsep{0em}
  \STATE Use \texttt{confidenceInterval} with seeds $(u_i)_{i=1}^R$ to obtain a confidence interval for each coordinate: $[\theta_L^{(i)},\theta_U^{(i)}]$ for $i=1,2,\ldots, d=\dim(\Theta)$. We assume that all endpoints are finite. Call $\Gamma_0 = \prod_{i=1}^d [\theta_L^{(i)},\theta_U^{(i)}]$.
  \STATE Set $b_r^{(i)} = \frac{\theta_U^{(i)} - \theta_L^{(i)}}{r}$, the grid width
  \STATE Set $\theta_j^{(i)} = \theta_L^{(i)} + j b_r^{(i)}$ for $j=0,1,2,\ldots, r$, the grid endpoints
  \STATE Call $\Gamma_0^r = \left\{\prod_{i=1}^d [\theta_{j_i}^{(i)},\theta_{j_i+1}^{(i)}] ~\middle |~ j_i \in \{0,1,2,\ldots, r-1\},~ i=1,\ldots, d\right\}$, the discretization of $\Gamma_0$ into a grid with resolution $r$
\end{algorithmic}
OUTPUT: $\bigcup\left\{\Theta_0\in \Gamma_0^r ~\middle |~ \texttt{accept}(\Theta_0,\alpha,(u_i)_{i=1}^R,G,s,T)=\texttt{TRUE}\right\}$
\label{alg:ConfGrid}
\end{algorithm}

\section{Hypothesis testing \texorpdfstring{$p$}{p}-values}\label{s:pvalue} 
While \citet{xie2022repro} focused on confidence intervals, they also highlight some basic connections to hypothesis testing, which we review here with our notation. { However, their $p$-value formula is purely theoretical and \citet{xie2022repro} do not even propose a Monte Carlo algorithm}. In contrast, we derive a { simulation-based}
$p$-value formula that leverages the results of Theorem \ref{thm:construction} and Lemma \ref{lem:confidence} to ensure conservative type I errors, even accounting for Monte Carlo errors.

We will consider a hypothesis test of the form $H_0: \theta^*\in \Theta_0$ versus $H_1:\theta^* \in \Theta_1$, where $\Theta_0\cap \Theta_1=\emptyset$. First, suppose we have a simple null hypothesis, $H_0:\theta^*=\theta_0$, where knowing $\theta_0$ fully specifies a generative model for the data. Let $B_\alpha(\theta_0;s,\omega)$ be an event such that $P_{s\sim F_{\theta_0},\omega\sim Q}(B_\alpha(\theta_0;s,\omega))\geq 1-\alpha$, just as in Section \ref{s:conf:abstract}. Define a rejection decision as $1-I(B_\alpha(\theta_0;s,\omega))$, which has type I error $\leq \alpha$. If we have an event $B_\alpha(\theta_0;s,\omega)$ defined for all $\alpha$, and these events are nested ($B_{\alpha_0}(\theta_0;s,\omega)\supset B_{\alpha_1}(\theta_0;s,\omega)$ when $\alpha_0\leq \alpha_1$), then a $p$-value is $p(\theta_0)=\inf \{\alpha|I(B_\alpha(\theta_0;s,\omega))=0\}$.

Now suppose that we have a composite null hypothesis $H_0:\theta^*\in \Theta_0$. Consider the rejection criteria $\inf_{\theta\in \Theta_0}[1-I(B_\alpha(\theta;s,\omega))]$, which we interpret that we only reject if { none of the $B_\alpha$ events occur for any $\theta\in \Theta_0$}. We see that this test also has type I error $\leq \alpha$. Note that $\sup_{\theta\in \Theta_0} p(\theta)$ is a $p$-value for $H_0: \theta^*\in \Theta_0$, where $p(\theta_0)$ is a $p$-value for $H_0: \theta^*=\theta_0$, which agrees with the $p$-value formula given in \citet[Corollary 1]{xie2022repro}.  If $B_\alpha(\theta;s,\omega)$ is constructed from an exchangeable sequence of test statistics, then this $p$-value can be expressed in a simple form in Theorem \ref{thm:pvalue}, accounting for the Monte Carlo errors. Furthermore, Algorithm \ref{alg:pval} gives a { pseudo code implementation of }this $p$-value.

\begin{restatable}{thm}{thmpvalue}\label{thm:pvalue}
Let $u_0,u_1,\ldots, u_R\iid P$, $\theta^*\in \Theta$ be the true parameter, $s=G(\theta^*,u_0)\in \RR^d$ be the observed value, $\theta\in\Theta_0\subset \Theta$, $s_i(\theta)=G(\theta,u_i)\in \mscr X$ be the repro samples, and $T_\theta:\RR^d\times \RR^{d\times (R+1)} \rightarrow \RR$ be a permutation-invariant statistic, where small values indicate that a sample is ``unusual''.  Call $T_{\text{obs}}(\theta) = T_\theta(s;s,s_1(\theta),\ldots, s_R(\theta))$, and let $T_{(1)}(\theta)$, $\ldots$, $T_{(R+1)}(\theta)$ be the order statistics of $T_\theta(s;s,s_1(\theta),\ldots, s_R(\theta))$, $T_\theta(s_1(\theta);s,s_1(\theta),\ldots, s_R(\theta)),\ldots,$\ $ T_\theta(s_R(\theta);s,s_1(\theta),\ldots,s_R(\theta))$.   Then, 
\[p=\frac{ \sup_{\theta\in \Theta_0} \#\{T_{(i)}(\theta)\leq T_{\text{obs}}(\theta)\}}{R+1},\]
is a $p$-value for the null hypothesis $H_0: \theta^*\in \Theta_0$.
\end{restatable}

{ In Algorithm \ref{alg:pval}, we give a simple procedure to calculate the $p$-value of Theorem \ref{thm:pvalue}, which in Proposition \ref{prop:pvalue} is proved to be correct.} 

\begin{restatable}{prop}{proppvalue}\label{prop:pvalue}
Under the same assumptions as Theorem \ref{thm:pvalue} and assuming that $T_\theta$ takes values in $[0,1]$, the output of Algorithm \ref{alg:pval} is equal to the $p$-value, stated in Theorem \ref{thm:pvalue}.
\end{restatable}

\begin{remark}
    In Algorithms \ref{alg:accept}-\ref{alg:pval}, the test statistic is restricted to take values in $[0,1]$. However, this is not much of a restriction, since any real-valued test statistic can be transformed to the interval $[0,1]$ using a cdf, such as $\Phi$. 
\end{remark}

\begin{remark}
    When implementing Algorithm \ref{alg:pval},  $\hat\theta(s)= \argmax_{\theta\in \Theta_0} [\#\{T_i(\theta)\leq T_{\text{obs}}(\theta)\}+T_{\text{obs}}(\theta)]$ may be used as an intermediate computation, but this value could have independent interest as an estimator. We can interpret that $\hat\theta(s)$ { is the parameter that makes the data the ``most plausible''}, since it maximizes the rank $\#\{T_{i}(\theta)\leq T_{\text{obs}}(\theta)\}$. Similarly, taking $\Theta_0=\Theta$, $\hat\theta(s)$ is an element of every confidence set, as the coverage goes to zero. We may also interpret $\hat\theta(s)$ as the mode of an implicit \emph{confidence distribution}; in the one-dimensional case, the mode of the confidence distribution is known to be a consistent estimator under some regularity conditions \citep[Section 4.2]{xie2013confidence}, and can be viewed as a generalization of maximum likelihood estimation.
\end{remark}

\begin{algorithm}[t]
\caption{\texttt{pvalue}($\Theta_0$, $(u_i)_{i=1}^R$, $G$, $s$, $T$)}
\scriptsize
INPUT: null hypothesis $\Theta_0$, seeds $(u_i)_{i=1}^R$, observed statistic $s\in \mscr X$, exchangeable statistic $T_\theta:\mscr X\times \mscr X^{R+1}\rightarrow [0,1]$ (low values give evidence against the null hypothesis), and $G(\theta,u)$ is the generating equation for $s$.

\begin{algorithmic}[1]
  \setlength\itemsep{0em}
  \STATE For a given $\theta$, denote $s_i(\theta) = G(\theta,u_i)$.
  \STATE Denote $T_i(\theta) = T_\theta(s_i(\theta);s,s_1(\theta),\ldots, s_R(\theta))$
  \STATE Denote $T_{\text{obs}}(\theta) = T_\theta(s;s,s_1(\theta),\ldots, s_R(\theta))$
    \STATE Set $M = \sup_{\theta\in \Theta_0}  [\#\{T_i(\theta)\leq T_{\text{obs}}(\theta)\}+T_{\text{obs}}(\theta)$]
    \STATE Set $p = \frac{1}{R+1} \min\{\lfloor M\rfloor+1,R+1\}$
\end{algorithmic}
OUTPUT: $p$ 
\label{alg:pval}
\end{algorithm}

\section{Simulations}\label{s:simulations}
In this section, we conduct multiple simulation studies, { applying our methodology to various privatized data settings,} 
with the following goals: 1) validate the coverage/type I errors of the simulation-based confidence intervals and hypothesis tests, 2) illustrate the improved performance of our methodology against other techniques such as parametric bootstrap, 3) demonstrate the flexibility of our inference methodology by considering a variety of models and privacy mechanisms, and 4) use our framework to understand how statistical inference is affected by clamping.  { In each simulation, we show that compared to the state-of-the-art methods, our methodology is able to offer comparable performance (power/confidence interval width) with more conservative type I error/coverage. Additional simulations are found in Section B of the supplementary materials.}


\subsection{Poisson distribution}\label{s:poisson}
{ In this section, we produce DP confidence intervals for Poisson data using our simulation-based methodology and investigate the effect of clamping on the confidence interval width.}

Suppose that $x_1,\ldots, x_n\iid \mathrm{Pois}(\theta^*)$, and we observe the privatized statistic $s = \frac 1n \sum_{i=1}^n [x_i]_0^c+(c/(n\ep))N$, for some noise distribution $N$, where $c$ is a fixed non-negative integer. Call $u^x_i\iid \mathrm{Unif}(0,1)$ and $u_x = (u^x_i)_{i=1}^n$ and $u_{DP}\overset d = N$. Then we have the transformation $x_i=F^{-1}(u^x_i;\theta)$, where $F^{-1}(\cdot;\theta)$ is the quantile function of $\mathrm{Pois}(\theta)$, and $G(\theta,(u_x,u_{DP})) = \frac 1n \sum_{i=1}^n[x_i]_0^c+\frac{c}{n\ep}u_{DP}$. 

To generate a simulation-based confidence interval { for $\theta$ }, we produce $R$ i.i.d. copies of $(u_x,u_{DP})$, and use the $(1-\alpha)$ sets according to part 1 of Example \ref{ex:ab} to form the confidence interval in Lemma \ref{lem:confidence}, { using $s$ itself as the test statistic}.
\begin{figure}[t]
    \centering
    \includegraphics[width=.48\linewidth]{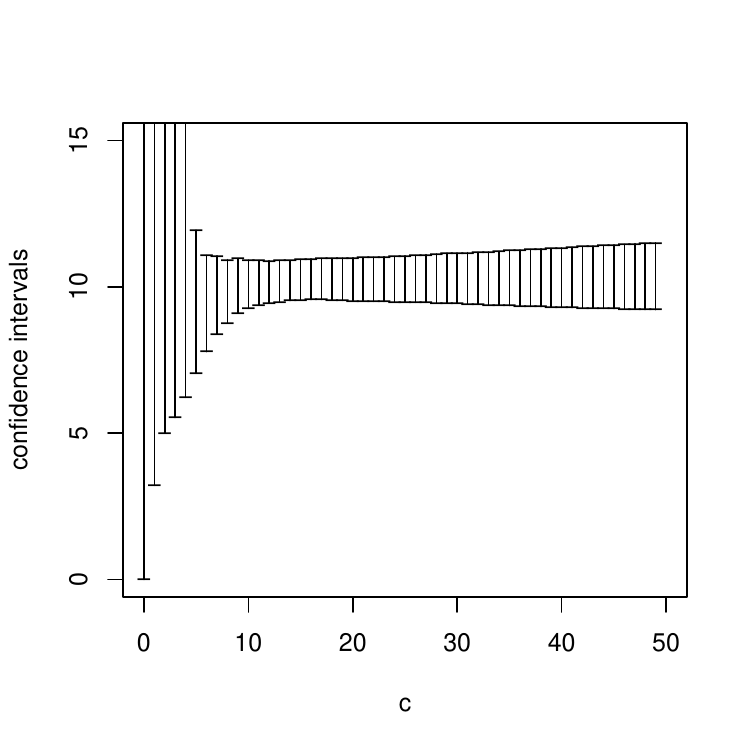}
    \includegraphics[width=.48\linewidth]{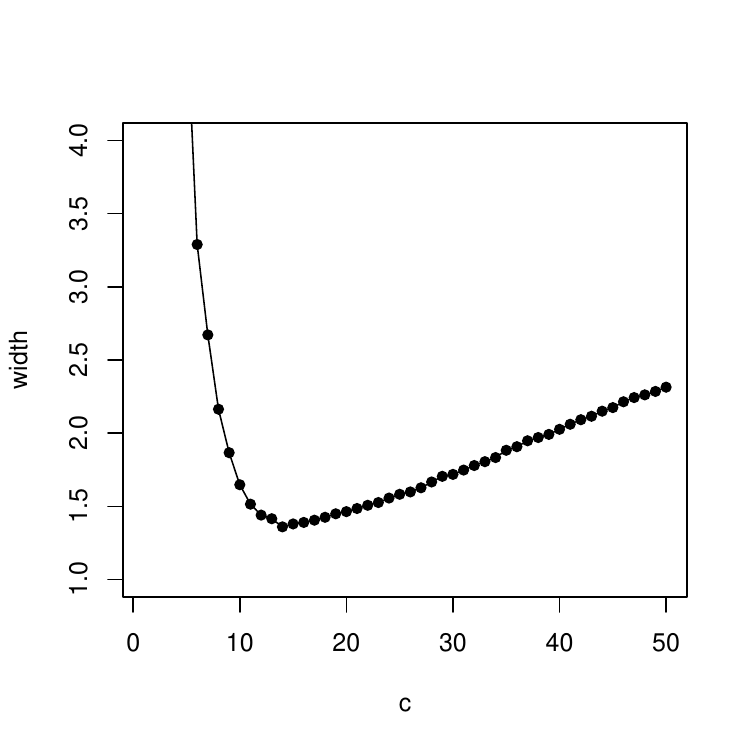}
    \caption{For fixed seeds in data generation, and in simulation-based inference, confidence intervals for $x_1,\ldots, x_{n} \sim \mathrm{Pois}(10)$ based on $s = \frac 1n \sum_{i=1}^n [x_i]_0^c+(c/(n\ep))N$, where $n=100$, $R=1000$, $N\sim N(0,1)$ and $\ep=1$. Left: $95\%$ confidence intervals as $c$ varies. Right: width of the $95\%$ confidence intervals as $c$ varies. For $c\leq 4$, the upper confidence limit is $\infty$. }
    \label{fig:poisson}
\end{figure}

In Figure \ref{fig:poisson}, we fix both the data generating seeds as well as the seeds for simulation-based inference, and compare the constructed confidence intervals (with provable $95\%$ coverage) as the quantity $c$ is varied. For the simulation, we set $n=100$, $\theta=10$, $\ep=1$, and $N\sim N(0,1)$, so that $s$ satisfies 1-GDP, and build our CIs from $R=1000$ repro samples. In the left plot of Figure \ref{fig:poisson}, we see the generated confidence intervals as $c$ varies. In the right plot of Figure \ref{fig:poisson}, we plot the widths of the intervals. We see that for $c\geq 15$, the width of the interval increases at an approximately linear rate, which reflects the fact that the privacy noise increases linearly in $c$. On the other hand, for $c\leq 10$, we see a rapid increase in the CI width when $c$ becomes smaller. In fact, in this simulation, for $c\leq 4$, the upper bound for the confidence interval is $+\infty$. 

Our simulation highlights the fact that while valid inference can be done with virtually any clamping threshold, there can be a much higher price to pay for choosing a threshold too low versus too high. Nevertheless, with $c=10$, approximately half of the data is clamped, and we are still able to get near minimal confidence intervals. For this simulation, the optimal clamping threshold is $c=14$, which alters about $8\%$ of the datapoints, giving a CI width of 1.46. 
In the right plot of Figure \ref{fig:poisson}, we see that even with $c=50$, we can get an interval with width 2.47, which is less than twice the optimal width. While this is a noticeable loss in accuracy, we are still able to make informative inference even with a very suboptimal choice of  $c$. See Section B.1 in the supplement for another example investigating the effect of the clamping threshold on confidence interval width.


\subsection{Location-scale normal} \label{s:normal}
{ In this section, we produce private confidence intervals for the location and scale parameters for normally distributed data, and compare the coverage and width against the parametric bootstrap \citep{ferrando2022parametric,du2020differentially} and the method of \citet{karwa2018finite}.}

Suppose that $x_1,\ldots, x_n\iid N(\mu^*,\sa^*)$, and we are interested in a confidence set for $(\mu^*,\sa^*)$. Suppose we are given clamping bounds $L$ and $U$, so that we work with the clamped data $[x_i]_{L}^U$. It is easy to see that the clamped mean $\overline{x_c}=\frac 1n \sum_{i=1}^n [x_i]_L^U$ has sensitivity $\frac{U-L}{n}$, and \citet{du2020differentially} derived the sensitivity of the sample variance $s^{(2)}_c=\frac{1}{n-1}\sum_{i=1}^n ([x_i]_{L}^U-\overline{x_c})^2$ is $\frac{(U-L)^2}{n}$, where the ``$c$" subscript reminds us that these statistics are for the clamped data. If $N_1,N_2\iid N(0,1)$, then a privatized statistic is $s = (\ol x_c+\frac{U-L}{n\ep} N_1, s^{(2)}_c+\frac{(U-L)^2}{n\ep}N_2)$, which satisfies $(\sqrt 2 \ep)$-GDP. { To see how this problem fits into our simulation-based inference, we sample $u^x_i \iid N(0,1)$ and use the transformation $x_i = \sa u_i^x+\mu$ to generate simulated datasets; for the DP mechanism, the seeds are the variables $N_1$ and $N_2$.}


\begin{table}[t]
    \centering
    \begin{tabular}{ll|ll}
    && $\mu$ & $\sigma$\\\hline
&    Empirical Coverage     & 0.989 (0.003)&  0.984 (0.004)\\
 Repro Sample &   Average Width    & 0.599 (0.003) &  0.756 (0.004)\\\hline
      &  Empirical Coverage     &0.688 (0.015)&0.003 (0.001) \\
 Parametric Bootstrap (percentile)&    Average Width    &0.311 (0.001)&0.291 (0.024)\\\hline
  &  Empirical Coverage     &0.859 (0.011)&0.819 (0.012)\\
 Parametric Bootstrap (simplified $t$)&    Average Width    &0.311 (0.001)&0.291 (0.024)
    \end{tabular}
    \caption{Nominal 95\% confidence intervals for location-scale normal. True parameters are $\mu^*=1$ and $\sigma^*=1$. Sample $x_1,\ldots, x_{100} \iid N(1,1)$, which are clamped to $[0,3]$. Gaussian noise with scale parameter $0.03$ is added to the sample mean of the clamped data, and Gaussian noise with scale parameter $0.09$ is added to the sample variance of the clamped data, jointly satisfying $\sqrt{2}$-GDP. The results are over 1000 replicates.}
    \label{tab:loc_scale_normal}
\end{table}

\begin{table}[t]
    \centering
    \begin{tabular}{l|ll}
    & Empirical Coverage & Average Width \\\hline
Repro Sample & { 0.989 (0.003) } & { 0.200 (0.001)}\\\hline
\citep{karwa2018finite} & { 0.994 (0.002)} & { 3.314 (0.013)}
    \end{tabular}
    \caption{Nominal 95\% confidence intervals for the location $\mu$ in location-scale normal. True parameters are $\mu^*=1$ and $\sigma^*=1$ where both are unknown. Sample $x_1,\ldots, x_{1000} \iid N(1,1)$, which are clamped to $[0,3]$. The results are under $1$-DP and computed over 1000 replicates. We use $R=200$ and Mahalanobis depth in repro samples.}
    \label{tab:loc_scale_normal_karwa}
\end{table}

For a simulation study, we generate 1000 replicates using  $n=100$, $R=200$, $\mu^*=1$, $\sigma^*=1$, $\epsilon=1$, $U=3$, and $L=0$. 
Using \eqref{eq:project} and Algorithm \ref{alg:CI} with the Mahalanobis depth function, we get $95\%$ (simultaneous) confidence intervals for $\mu$ and $\sigma$ with average coverage 0.989 and width 0.599 for $\mu$ and coverage 0.984 and width 0.756 for $\sigma$, as shown in Table \ref{tab:loc_scale_normal}. 

We compare against the parametric bootstrap, using the estimator $\hat\theta=(s^{(1)}, \sqrt{\max\{s^{(2)},0\}})$ and 200 bootstrap samples; note that the parametric bootstrap targets two marginal $95\%$ confidence intervals, whereas our method gives 95\% simultaneous intervals. Using the percentile method parametric bootstrap gives average coverage of 0.688 and width 0.311 for $\mu$, and coverage of 0.003 and width 0.291 for $\sigma$ (over 1000 replicates). While the parametric bootstrap intervals are much smaller, the coverage is unacceptably low. The parametric bootstrap can be improved to some extent by using a simplified version of the bootstrap-$t$ interval, which is based on the empirical distribution of $2\hat\theta-\hat\theta_b$, where $\hat\theta_b$ is the value from the $b^{th}$ bootstrap; this method offers some bias correction and improves the coverage to 0.859 and 0.819 for $\mu$ and $\sigma$, respectively. While more complex bias-correction methods could potentially improve the parametric bootstrap, many available methods are inapplicable for privatized data, since the original data is unobserved. 

{ We also compare our approach against the confidence interval of \citet{karwa2018finite}. However, the Algorithm 5 of \citet{karwa2018finite} satisfies $\ep$-DP rather than Gaussian DP, so to have a fair comparison, we change the distributions of $N_1,N_2$ above to be $\mathrm{Laplace}(0,1/2)$ so that our summary statistics also satisfy $\ep$-DP with $\ep=1$. For the simulation study, we use the same settings as above, except we only produce confidence intervals for $\mu$ and we set confidence level $1-\alpha=0.95$ and $n=1000$, since the method of \citet{karwa2018finite} only targets $\mu$ and does not perform well for small $n$. Following the usage of Algorithm 5 of \citet{karwa2018finite} by \citet{du2020differentially}, we set $\alpha_0=\alpha_1=\alpha_2=\alpha_3=\alpha/2$, $\ep_1=\ep_2=\ep_3=\ep/3$. We set $(\mu_{\min},\mu_{\max})=(-10, 10)$ and $(\sigma_{\min},\sigma_{\max})=(10^{-8}, 10)$ for both approaches. The results are found in Table \ref{tab:loc_scale_normal_karwa}, where we see that both confidence interval methods have conservative coverage compared to the $95\%$ nominal level, but the simulation-based method has a significantly shorter average width of $0.2$ compared to $3.3$.}

\subsection{Simple linear regression hypothesis testing}\label{s:regression}

\begin{figure}
    \centering
    \includegraphics[width=.9\linewidth]{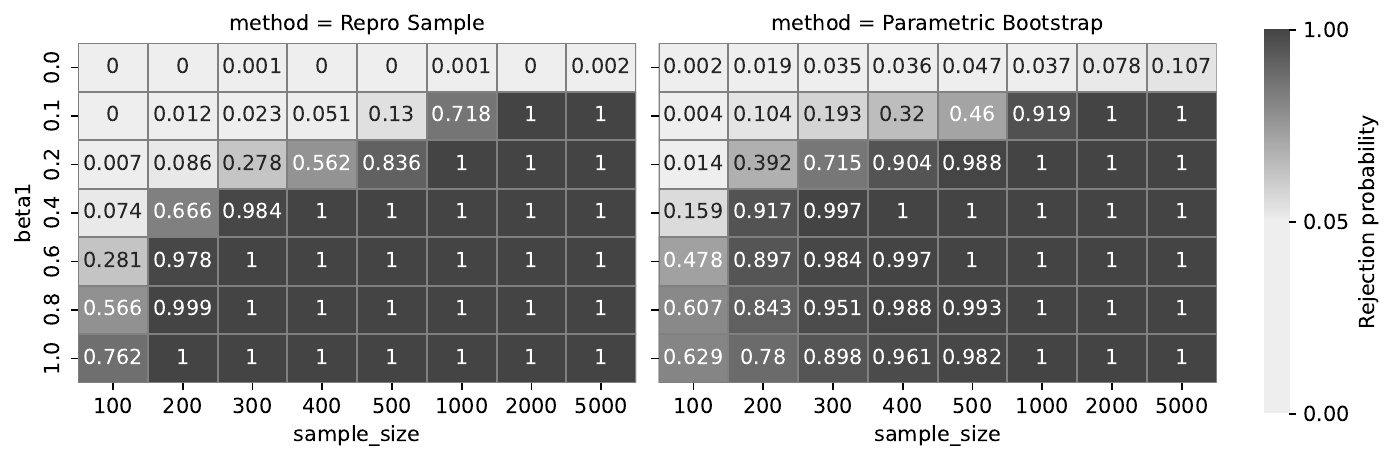}
    \caption{The rejection probability for hypothesis testing on $H_0:\beta_1^*=0$ and $H_1:\beta_1^*\neq 0$ in a linear regression model $Y=\beta_0^*+X\beta_1^* + \epsilon$ with repro and parametric bootstrap \citep{alabi2022hypothesis}. The significance level is 0.05, and the values in the table are calculated from 1000 replicates.}
    \label{fig:LR-HT-gdp1}
\end{figure}

This example follows the setting in \citep{alabi2022hypothesis}, and we compare their parametric bootstrap method, Algorithm 2 in \citep{alabi2022hypothesis}, with our simulation-based method. Consider the model $Y=\beta_0^* +X\beta_1^* +  \epsilon$ where we want to test $H_0:\beta_1^*=0$ and $H_1:\beta_1^*\neq 0$. We assume $x_i\sim N(\EE[X], \mathrm{Var}(X)), \ep_i\sim N(0, \mathrm{Var}(\ep)), y_i= \beta_0^* +x_i\beta_1^* +  \epsilon_i$, and we let $\theta^*=(\beta_1^*, \beta_0^*, \EE[X], \mathrm{Var}(X), \mathrm{Var}(\ep))$ which fully parameterizes the model; { the parameters other than $\beta_1^*$ are viewed as nuisance parameters}. Both the parametric bootstrap and our simulation-based method use the sufficient statistic perturbation with Gaussian mechanism: Let $\tilde{\bar x}=\frac{1}{n} \sum_{i=1}^n [x_i ]_{-\Delta}^{\Delta}+\frac{2\Delta}{ (\mu/\sqrt{5}) n}N_1$, $\widetilde{\overbar {x^2}}=\frac{1}{n}\sum_{i=1}^n {[x^2_i ]_{0}^{\Delta^2}}+\frac{\Delta^2}{ (\mu/\sqrt{5}) n}N_2$, $\tilde{\bar y}=\frac{1}{n}\sum_{i=1}^n { [y_i ]_{-\Delta}^{\Delta}}+\frac{2\Delta}{ (\mu/\sqrt{5}) n}N_3$, $\widetilde{\overbar {xy}}=\frac{1}{n}\sum_{i=1}^n {[x_i y_i]_{-\Delta^2}^{\Delta^2}}+\frac{2\Delta^2}{ (\mu/\sqrt{5}) n}N_4$, $\widetilde{\overbar {y^2}}=\frac{1}{n}\sum_{i=1}^n{[y_i^2]_{0}^{\Delta^2}}+\frac{\Delta^2}{ (\mu/\sqrt{5}) n}N_5$ where $N_i\iid N(0,1)$; releasing $s=(\tilde{\bar x}, \widetilde{\overbar {x^2}}, \tilde{\bar y}, \widetilde{\overbar {xy}}, \widetilde{\overbar {y^2}})$ is guaranteed to satisfy $\mu$-GDP.  Using the released $s$, \citet{alabi2022hypothesis} estimate $\theta$ and calculate the observed test statistic, then perform parametric bootstrap under the null hypothesis to compare the observed test statistic with the generated test statistics. In contrast, our Algorithm \ref{alg:pval} finds the $\theta$ such that $s$ fits best into $(s_i(\theta))_{i=1}^R$ with respect to its Mahalanobis depth, and  we compare the corresponding $p$-value with the significance level.\footnote{To reduce the computational cost, we can stop Algorithm \ref{alg:pval} if there exists one $\theta\in \Theta_0$ with $p = \frac{1}{R+1}[\#\{T_i(\theta)\leq T_{\text{obs}}(\theta)\}+1]$ greater than the significance level and state ``do not reject the null hypothesis.''}. { Similar to Section \ref{s:normal}, we use several $N(0,1)$ random variables as the seeds for the data, which we transform using a linear map to produce the data $(x_i,y_i)$; the privacy seeds are $N_1,\ldots, N_5$.}

\begin{figure}
    \centering
    \includegraphics[width=.9\linewidth]{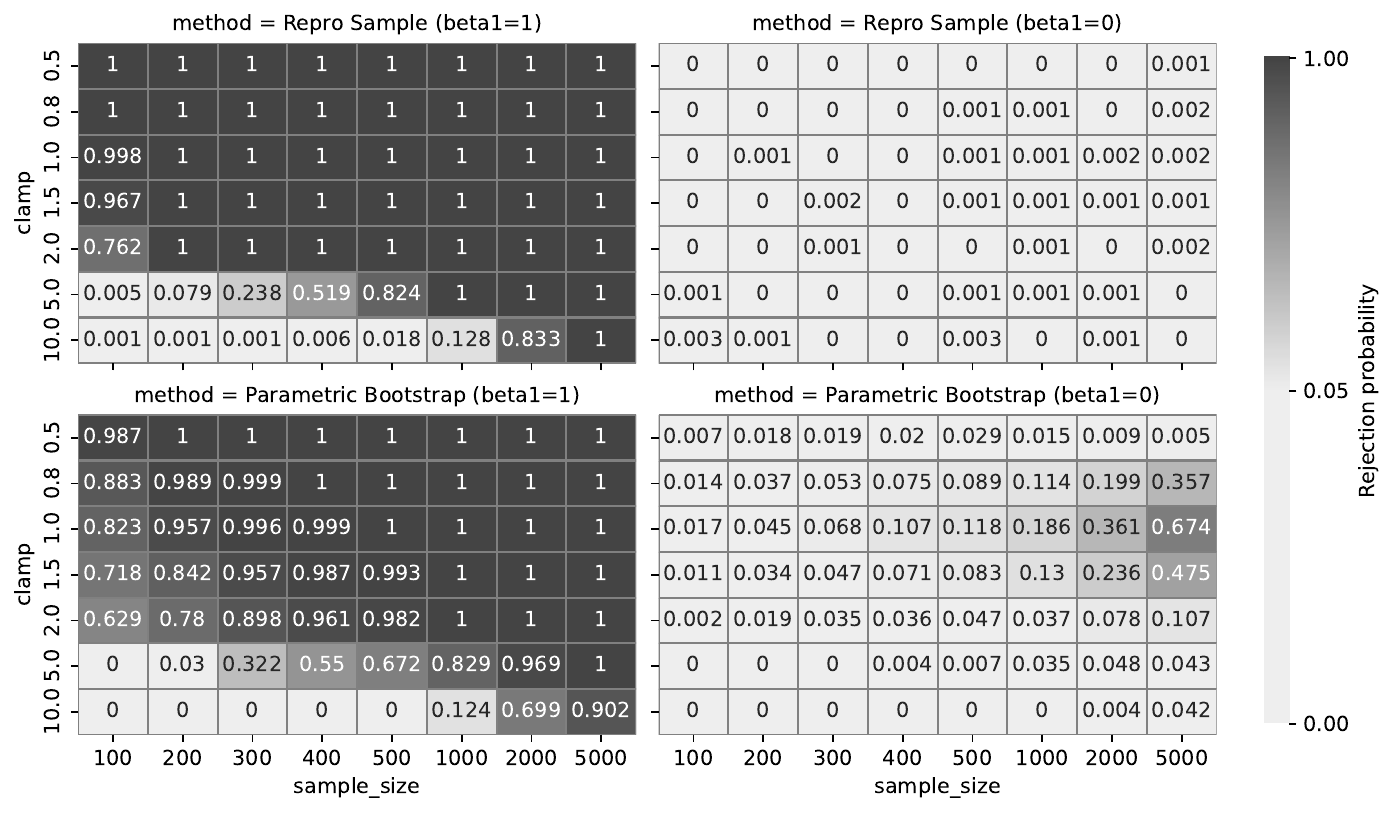}
    \caption{The rejection probability for hypothesis testing on $H_0:\beta_1^*=0$ and $H_1:\beta_1^*\neq 0$ in a linear regression model $Y=\beta_0^* + X\beta_1^* + \epsilon$ with repro and parametric bootstrap. The settings are the same as in Figure \ref{fig:LR-HT-gdp1} except for the clamping region.}
    \label{fig:LR-HT-compareclamp}
\end{figure}

Our simulation results in Figure \ref{fig:LR-HT-gdp1} show the rejection probability for $\beta_1^*=0$, 0.1, 0.2, 0.4, 0.6, 0.8, 1 under the settings of the sample size $n=100$, 200, 300, 400, 500, 1000, 2000, 5000, $R=200$, $\beta_0^*=-0.5$, $x_i \sim N(0.5, 1)$, $\ep_i \sim N(0, 0.25)$, the clamping range $\Delta=2$, the significance level $\alpha=0.05$, and the privacy guarantee $1$-GDP. The rejection probability by our simulation-based method is similar to the parametric bootstrap method except that when $\beta_1^*=0$ and $n\geq 2000$, the parametric bootstrap fails to control the type I error below the significance level of $0.05$, whereas our repro sample maintains conservative type I errors.

The success of our simulation-based method is natural because of our finite sample guarantee in Theorem \ref{thm:pvalue}. Although \citet{alabi2022hypothesis} proved in their Theorem 5.1 that their DP F-statistic converged in distribution to the asymptotic distribution of the non-private F-statistic, one of their assumptions is that $P(\exists i\in[n], y_i \notin [-\Delta_n, \Delta_n]) \rightarrow 0 \text{ and } \forall i\in[n],\ x_i \in[-\Delta_n, \Delta_n]$ as $n\rightarrow\infty$, { which may be not achievable since $\Delta_n$ requires knowing the true parameters and every $x_i$. Without such knowledge, it is likely that $\Delta_n$ is chosen incorrectly and the clamping procedure effectively changes the data distribution. Furthermore, even if $\Delta_n$ is chosen correctly, the results of Alabi and Vadhan (2022) are only asymptotic, lacking finite-sample guarantees. }
We empirically verify this limitation by our simulation results in Figure \ref{fig:LR-HT-compareclamp}, which shows the rejection probability for $\Delta=0.5$, 0.8, 1, 1.5, 2, 5, 10. In the bottom right subfigure of Figure \ref{fig:LR-HT-compareclamp}, the parametric bootstrap drastically fails to maintain the significance level $0.05$ when $n\geq 2000$ and $\Delta=0.8$, 1, 1.5, 2; what is especially concerning is that the type I errors seem to get worse, rather than better, as the sample size increases. On the other hand, the type I errors for our simulation-based approach are very well controlled, never exceeding 0.004. In addition to having better-controlled type I errors, our method also often has higher power, especially in smaller sample sizes. 
Additional simulations, where the privacy parameter $\mu$ is varied, are found in Section B.2 of the supplement.

\subsection{Logistic regression via objective perturbation
}\label{s:logistic}
In this section, we demonstrate that our methodology can be applied to non-additive privacy mechanisms as well. In particular, we derive confidence intervals for a logistic regression model using the objective perturbation mechanism \citep{chaudhuri2008privacy}, via the formulation in \citep{awan2021structure}, which adds noise to the gradient of the log-likelihood before optimizing, resulting in a non-additive privacy noise. However, since we assume that the predictor variables also need to be protected, we privatize their first two moments via the optimal $K$-norm mechanism \citep{hardt2010geometry,awan2021structure}, a multivariate additive noise mechanism. Details for both mechanisms are in Section B.3 of the supplement. { We compare the performance of our confidence intervals against those developed in \citet{wang2019differentially}, which use a custom privacy mechanism, and a combination of asymptotics and the parametric bootstrap.}

We assume that the predictor variable $x$ is naturally bounded in some known interval, and normalized to take values in $[-1,1]$. We model $x$ in terms of the beta distribution: $x_i=2z_i-1$, where $z_i\iid \mathrm{Beta}(a^*,b^*)$. Then $y|x$ comes from a logistic regression: $y_i|x_i\sim \mathrm{Bern}(\mathrm{expit}(\beta_0^*+\beta_1^* x_i))$, 
where $\mathrm{expit}(x) = \exp(x)/(1+\exp(x))$. In this problem, the parameter is $\theta^* = (\beta_0^*,\beta_1^*,a^*,b^*)$, and we are interested in producing a confidence interval for $\beta_1^*$; { the other parameters are viewed as nuisance parameters}.

To set up the generating equation, we use inverse transform sampling for $z_i$: $z_i = F^{-1}_{a^*,b^*}(u^z_i)$, where $F_{a^*,b^*}^{-1}$ is the quantile function of $\mathrm{Beta}(a^*,b^*)$, and $u^z_i \iid U(0,1)$; similarly, we generate $y_i = I(u^y_i\leq \mathrm{expit}(\beta_0^*+\beta_1^*x_i))$, where $u^y_i \iid U(0,1)$. The privatized output $s$ consists of estimates of $(\beta_0^*,\beta_1^*)$ from objective perturbation, as well as noisy estimates of the first two moments of the $z_i$'s, privatized by a $K$-norm mechanism. Since the only source of randomness in the two privacy mechanisms are independent variables from $K$-norm distributions, we use the $K$-norm random variables as the ``seeds'' for the privacy mechanisms. See Section B.2 of the supplement for details about objective perturbation and $K$-norm mechanisms. 

In the simulation generating the results in Figure \ref{fig:logistic}, we let $a^*=b^*=0.5$, $\beta_0^*=0.5$, $\beta_1^*=2$, $R=200$, $\alpha=0.05$, $n=100,$ 200, 500, 1000, 2000, and $\ep=0.1$, 0.3, 1, 3, 10 in $\ep$-DP. Other details for this experiment are found in Section B.3 of the supplement.

\begin{figure}
    \begin{minipage}[t]{.62\textwidth}
        \centering
        \includegraphics[width=\textwidth]{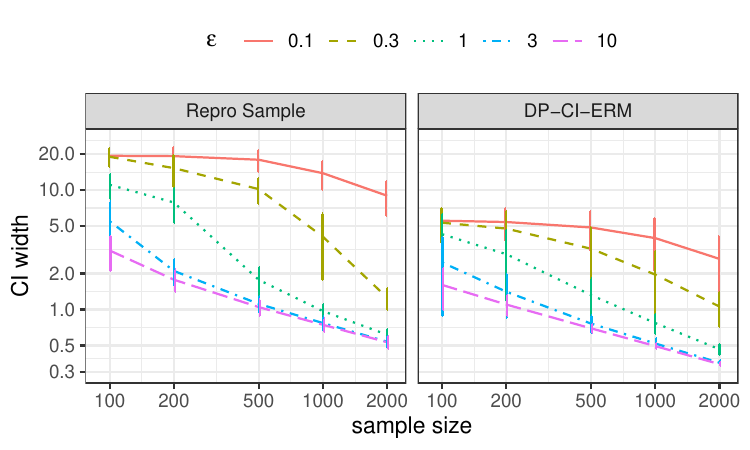}
    \end{minipage}
    \hfill
    \begin{minipage}[t]{.38\textwidth}
        \centering
        \includegraphics[width=\textwidth]{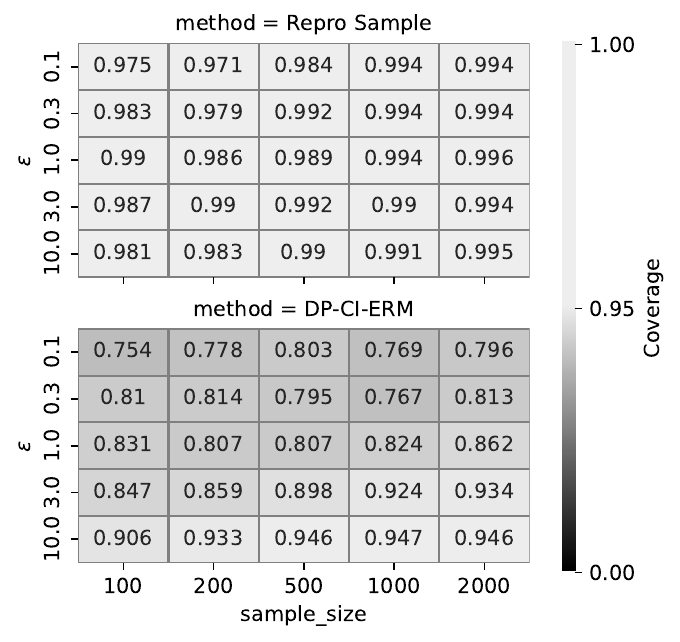}
    \end{minipage}  
    \caption{Width and coverage for the confidence intervals of  $\beta_1$ in logistic regression with repro and DP-CI-ERM \citep{wang2019differentially}. Parameters for this simulation are $a^*=b^*=0.5$, $\beta_0^*=0.5$, $\beta_1^*=2$, $R=200$, $\alpha=0.05$, and the results were averaged over 1000 replicates.}
    \label{fig:logistic}
\end{figure}

In Figure \ref{fig:logistic}, we compare the width and coverage of our simulation-based confidence intervals using Mahalanobis depth, against an alternative method proposed by \citet{wang2019differentially}, which develops a joint privacy mechanism and confidence interval algorithm for empirical risk minimization problems, and which is also based on objective perturbation (we refer to this method as DP-CI-ERM). Some key differences between our simulation-based approach and DP-CI-ERM are 1) our approach allows for arbitrary privacy mechanisms, whereas DP-CI-ERM requires one to use a specific privacy mechanism, 2) currently our method requires a fully parametric model, whereas DP-CI-ERM only needs the empirical risk function, and 3) our method gives finite sample coverage guarantees, whereas DP-CI-ERM gives an asymptotic coverage guarantee. Because of these differences, this comparison is somewhat apples-to-oranges, but is still instructive to understand the pros and cons of the repro methodology. We see in Figure \ref{fig:logistic} that DP-CI-ERM generally has smaller confidence interval width compared to repro, considering a variety of sample sizes and privacy budgets -- roughly, the widths for repro are about 2-4 times as wide as those for DP-CI-ERM. However, while the coverage of DP-CI-ERM is very close to 0.95 at large sample sizes or large $\ep$ values, we see that it has significant under coverage in several settings of $n$ and $\epsilon$. On the other hand, we see that the simulation-based confidence intervals always have greater than 95\% coverage.

\section{Discussion}\label{s:discussion}
We have shown that simulation-based inference is a powerful technique that enables finite-sample frequentist confidence intervals and hypothesis tests on privatized data. We have shown through several examples that our methodology offers more reliable, and sometimes even more powerful inference compared to existing methods, such as the parametric bootstrap. We have also improved the original repro methodology to account for Monte Carlo errors, and emphasize that while our results were motivated by problems in differential privacy, our inference framework is also applicable to models outside of privacy.

A limitation of our simulation-based methodology is that it is generally not guaranteed that the resulting confidence set will be connected, and due to this limitation, { the confidence interval found by Algorithm \ref{alg:CI} may have inflated width.} Future researchers may investigate necessary and sufficient conditions for the repro confidence sets to be connected as well as { other techniques to reduce the width of the simulation-based confidence intervals}. 

{ In our simulations, we found that the parametric bootstrap method of \citet{ferrando2022parametric} often produced unacceptable coverage/type I errors. Based on \citet{beran1997diagnosing}, we know that bootstrap consistency typically requires \emph{asymptotic equivariance} of the estimator, which itself requires the estimator to be consistent and asymptotically normal. Due to clamping, many DP estimators are biased (and generally inconsistent when the clamping bounds are fixed for all $n$), and this is likely the cause for the failure of the parametric bootstrap in these settings. Thus, it may be possible to improve the performance of the parametric bootstrap for privatized data by deriving estimators with lower bias. Nevertheless, at best the parametric bootstrap gives an asymptotic approximation while our framework has finite-sample guarantees.}


\subsection{Choice of test statistic}\label{s:cost}
The ideal test statistic for use in repro would be a pivot, whose formula and sampling distribution do not depend on the nuisance parameters, as this avoids the issue of over-coverage as well as reduces the computation burden of optimizing over the nuisance parameters. However, finding a pivot is not always possible. Especially in the setting of differential privacy, it is very challenging to design test statistics whose distribution does not depend on nuisance parameters (e.g., \citet{gaboardi2016differentially,du2020differentially,awan2023canonical,alabi2022hypothesis}). As \citet{xie2022repro} suggested, the goal is then to find an approximate pivot, such as one whose asymptotic distribution does not depend on the nuisance parameters. While classical statistical methods suggest plugging in an estimator for the nuisance parameters to approximate the sampling distribution using asymptotic theory, there is no finite sample guarantee that this will give accurate coverage/type I error rates. On the other hand, using repro samples, we can still get the benefit of using an approximate pivot, while still ensuring coverage/type I error rates. See Section B.4 of the supplement for an example where we construct a custom test statistic and demonstrate the improved coverage compared to the depth statistic. A problem for future work would be to develop general strategies to construct approximate pivots from DP statistics, which can both reduce the over-coverage and optimize the test statistic for the model and mechanism at hand.

On the other hand, most of the examples considered in this paper used a depth statistic as the default test statistic in repro. In this section, we quantify the suboptimality of this approach in terms of increased width/decreased power. There are two main issues at play: 1) The depth statistic may generally be a sub-optimal test statistic and 2) Because we are implicitly obtaining a simultaneous confidence set for all $d$ parameters, assuming that the marginal coverages of each projected confidence set are roughly equal, we would expect each coverage to be $(1-\alpha)^{1/d}$, resulting in an increased width.

 While there is often a loss in power from using a default test statistic, such as a depth statistic, it is difficult to quantify the suboptimality, as this is problem specific.


\begin{table}[t]
    \centering
    \begin{tabular}{c|cccccc}
        $d$ &1& 2&5&10&100&1000 \\\hline
        Relative width &1& 1.14&1.31&1.43&1.77&2.07
    \end{tabular}
    \caption{Relative width due to over-coverage for the normal mean with known variance, when the nominal level is $1-\alpha=0.95$, and the over-coverage level is $1-\alpha^* = (1-\alpha)^{1/d}$. }
    \label{tab:overCoverage}
\end{table}

To understand the cost of over-coverage, we consider as an idealized setting the widths of confidence intervals on normal data -- this is a useful example since many statistics are asymptotically normal. 
We explore the increased width due to over-coverage when $\alpha=0.05$ and $d$ increases in Table \ref{tab:overCoverage}. We see that having a single nuisance parameter increases the width by 14\%, but that as $d$ increases, the increase in the width is very slow, requiring $d=1000$ before we have about twice the width. So, while over-coverage does increase the width of our intervals, the increased width is often within a factor of 1.5-2, even for a large number of nuisance parameters. For a fixed nominal level $\alpha$, the relative width of the normal confidence interval grows proportionally to $\sqrt{\log(d)}$, a very slow rate (see Proposition A.4 in the supplement).

{
\subsection{Non-parametric inference}
As a simulation-based inference method, repro requires the function $G$ to generate $s$ using $s=G(\theta, u)$.  While the previous examples shown in the paper have a parametric model for the sensitive dataset $D$, there are also examples where the generation of $s$ does not require a parametric model of $D$.  

For example, to build DP confidence intervals for the median of a dataset following an unknown distribution, it is possible to design an $s$ which can be simulated directly without specifying a model for $D$. Given a dataset of real-values, $D=(x_1, \ldots, x_n)$ where $x_i$ are i.i.d. from an unknown distribution $F$, we define $s = \sum_{i=1}^n I(\{x_i \leq \mathrm{median}(F) \})$ and simulate $s$ using $s \sim \mathrm{Binom}(n, 0.5)$, which is proposed by \citet[Example 2(A)]{xie2022repro} in making inference about the population quantile $\theta_0 = F^{-1}(0.5)$. If we follow this idea to use repro to build DP confidence intervals for the median, we could potentially obtain a method similar to the \texttt{CDFPostProcess} method by \citet{drechsler2022nonparametric}, by adding noise to $s$ on a grid of candidate values for $\mathrm{median}(F)$ (note though that \citet{drechsler2022nonparametric} derived their result without considering the repro framework). 

Similar to building DP confidence interval for the median, we can use repro in conducting DP non-parametric hypothesis tests such as the Mann-Whitney test. In our supplementary materials Section B.5, we explain how to use repro in DP hypothesis test based on the DP version of Mann-Whitney test statistic, and we compare repro to two state of the art methods, \citep{couch2019differentially} and \citep{kazan2023test}. We show through simulations that our method has a better or comparable power to their methods. 

We leave it as open problems how to construct such summaries which have a parametric form, especially in privatized data settings, as well as the development of other techniques when no such parametric summary is available.
}



{
\bibliographystyle{plainnat} 
\bibliography{bibliography.bib} }

\appendix

\section{Proofs and technical details}

In this section, we present the proofs and technical details needed to support the results in the main paper.

Results similar to Lemmas \ref{lem:OrderInterval} and \ref{lem:domUnif} have been used in conformal prediction to guarantee coverage rates for the predictions. We include proofs of the lemmas for completeness. 

\begin{lemma}\label{lem:OrderInterval}
Let $x_1,\ldots, x_n$ be an exchangeable sequence of real-valued random variables. Then $P(x_{(a)}\leq x_n)\geq \frac{n-a+1}{n}$, $P(x_n\leq x_{(b)})\geq \frac{b}{n}$, and $P(x_{(a)}\leq x_n\leq x_{(b)})\geq \frac{(b-a+1)}{n}$, where $x_{(a)}$ and $x_{(b)}$ are the $a^{th}$   and $b^{th}$ order statistics of $x_1,\ldots, x_n$. 
Furthermore, for $a_\alpha = \lfloor \alpha n \rfloor+1$, $P(x_{(a_\alpha)}\leq x_n)\geq 1-\alpha$, and for $b_{\alpha} = \lceil (1-\alpha)n\rceil$, $P(x_n\leq x_{(b_\alpha)})\geq 1-\alpha$. 
For any $a=1,2,\ldots, n-\lceil(1-\alpha)n\rceil+1$, setting $b=a+\lceil (1-\alpha)n\rceil -1$ ensures that $P(x_{(a)}\leq x_n\leq x_{(b)})\geq 1-\alpha$. 
 \end{lemma}
\begin{proof}
Notice that $x_n$ is equally likely to take any of the $n$ positions among the order statistics. In the first case, there are $a-1$ positions not in the interval (assuming no ties), so there are $n-a+1$ positions that make the probability true. In the second case, there are $b$ positions that make the probability true, and in the third case, there are  $(b-a+1)$ positions in the interval. If there are ties at the boundaries, then each of these counts only increases. 

For $a_\alpha$ and $b_\alpha$ defined above, note that 
$P(x_{(a_\alpha)}\leq x_n)\geq \frac{n-(a_\alpha-1)}{n}= \frac{n-\lfloor \alpha n \rfloor}{n}\geq 1-\alpha,$
and 
$P(x_n\leq x_{(b_\alpha)})\geq \frac{b_\alpha}{n}=\frac{\lceil (1-\alpha)n\rceil}{n}\geq 1-\alpha.$ 
Finally, let $a$ and $b$ be any values such that $b-a=\lceil(1-\alpha)n\rceil -1$. Then by the earlier result, we have
$P(x_{(a)}\leq x_n\leq x_{(b)})\geq (b-a+1)/n=\lceil(1-\alpha)n\rceil/n\geq (1-\alpha).$\qedhere
\end{proof}

\begin{lemma}\label{lem:domUnif}
Let $ x_1,\ldots, x_n$ be an exchangeable sequence of real-valued random variables. Then $p = n^{-1} \#\{x_i\leq x_n\}$ satisfies 
$P(p\leq \alpha)\leq \alpha$. 
\end{lemma}
\begin{proof}
First, let $k\in\{0,1,2,\ldots, n-1\}$, and consider 
\begin{align*}
    P(\#\{x_i\leq x_n\}\leq k)
    &=P(\#\{x_i\leq x_n)<k+1)\\
	&= P(x_n<x_{(k+1)})\\
    &=1-P(x_n\geq x_{(k+1)})\\
    &\leq 1-(n-(k+1)+1)/n\\
    &= k/n.
\end{align*}

where we used Lemma \ref{lem:OrderInterval} to establish the inequality. Now, setting $k=\lfloor n\alpha\rfloor$, we have that 
    \[P(p\leq \alpha) = P(\#\{x_i\leq x_n\}\leq n\alpha)
    =P(\#\{x_i\leq x_n\}\leq \lfloor n\alpha\rfloor)
    =\lfloor n\alpha\rfloor/n
    \leq \alpha.\qedhere\]
\end{proof}

\begin{proof}[Proof of Lemma 3.1.]
The fact that $\Gamma_\alpha(s,\omega)$ is a $(1-\alpha)$-confidence set is easily seen from the following: $P_{s\sim F_\theta,\omega\sim Q}(\theta\in \Gamma_\alpha(s,\omega))
=P_{s\sim F_\theta,u\sim Q}(B_\alpha(\theta;s,\omega))\geq 1-\alpha.$ 
Note that $\Gamma^\beta_\alpha(s,\omega)\times \{\eta|\exists \beta^* \text{ s.t. } (\beta^*,\eta) \in \Theta\}\supset \Gamma_\alpha(s,\omega)$. It follows that 
\[P_{s\sim F_{(\beta,\eta)},\omega\sim Q}(\beta \in \Gamma_\alpha^\beta(s,\omega))
\geq P_{s\sim F_{(\beta,\eta)},\omega\sim Q}((\beta,\eta) \in \Gamma_\alpha(s,\omega))\geq 1-\alpha.\]
To see that $\Gamma_\alpha^{\theta_1}(s,\omega),\ldots, \Gamma_{\alpha}^{\theta_k}(s,\omega)$ are simultaneous $(1-\alpha)$-confidence sets for $\theta_1,\ldots, \theta_k$, consider 
\begin{align*}
P_{s\sim F_{(\theta_1,\ldots, \theta_k)},\omega\sim Q}(\theta_1\in \Gamma_\alpha^{\theta_1}(s,\omega),\ldots, \theta_k \in \Gamma_\alpha^{\theta_k}(s,\omega))
&=P_{s\sim F_\theta,\omega \sim Q}(\theta \in \Gamma_\alpha^{\theta_1}(s,\omega)\times \cdots \times \Gamma_\alpha^{\theta_k}(s,\omega))\\
&\geq P_{s\sim F_\theta,\omega \sim Q}(\theta \in \Gamma_\alpha(s,\omega))\\
&\geq 1-\alpha,
\end{align*}
since $\Gamma_\alpha^{\theta_1}(s,\omega)\times \cdots \times \Gamma_\alpha^{\theta_k}(s,\omega)\supset \Gamma_\alpha(s,\omega)$. 
\qedhere
\end{proof}

\begin{proof}[Proof of Theorem 3.6.]
The fact that $B_\alpha(\theta;s,(u_i)_{i=1}^R)$ satisfies (1) follows from Lemma \ref{lem:OrderInterval}, since $T_\theta(s;s,s_1(\theta),\ldots, s_R(\theta)),$ $T_\theta(s_1(\theta);s,s_1(\theta),\ldots, s_R(\theta)),\ldots, T_\theta(s_R(\theta);s,s_1(\theta),\ldots,s_R(\theta))$ are exchangeable. { Setting $\omega = (u_i)_{i=1}^R$ and $Q = P^R$,} it follows from Lemma 3.1 that both $\Gamma_\alpha(s,(u_i)_{i=1}^R)$ and $\Gamma_\alpha^\beta(s,(u_i)_{i=1}^R)$ are $(1-\alpha)$-confidence sets. 
\end{proof}

{
\begin{lemma}
    \label{lem:accept}
    Suppose that the inputs to Algorithm 1 are as specified, and call $\Gamma$ the confidence set of Theorem 3.6 using $a=\lfloor \alpha(R+1)\rfloor +1$ and form (2) of Lemma 3.1. Then, Algorithm 1 returns \texttt{TRUE} if and only if $\Theta_0\cap \Gamma\neq \emptyset$. 
\end{lemma}
\begin{proof}
We consider alternative expressions of $\Theta_0\cap \Gamma:$ 
\begin{align*}
    \Theta_0 \cap \Gamma 
    &= \{\theta\in \Theta_0 | T_{obs}(\theta) \geq T_{(a)}(\theta)\}\\
    &=\{\theta\in \Theta_0 | \#\{T_{obs}(\theta)\geq T_{(i)}(\theta)\}\geq a\}\\
    &=\{\theta\in \Theta_0 | \#\{T_{obs}(\theta)\geq T_i(\theta)\}+1\geq a \}\\
    &=\{\theta\in \Theta_0 | \#\{T_{obs}(\theta)\geq T_i(\theta)\}+1+T_{obs}(\theta) \geq a \},
\end{align*}
where $T_{(i)}(\theta)$ is the notation from Theorem 3.6 and $T_i(\theta)$ is the notation from Algorithm 1; the final equation follows because $T_{obs}(\theta)\in [0,1]$, $a\in \ZZ$, and $T_{obs}(\theta)=1$ implies that $\#\{T_{obs}(\theta)\geq T_{(i)}(\theta)\}=R+1$. 

From the above equations, we see that $\Theta_0\cap \Gamma \neq \emptyset$ if and only if $\sup_{\theta\in \Theta_0} \#\{T_{obs}(\theta)\geq T_i(\theta)\}+1+T_{obs}(\theta) \geq a$, which is precisely when Algorithm 1 returns \texttt{TRUE}.
\end{proof}

\begin{proof}[Proof of Proposition 3.11.]
We will focus on the left portion of the confidence interval, as the argument for the right side is symmetric. Lemma \ref{lem:accept} showed that Algorithm 1 returns \texttt{TRUE} if and only if the input set has intersection with the confidence set of Theorem 3.6. For brevity, call $\Gamma^\beta$ the confidence set of Theorem 3.6 using form (3) of Lemma 3.1. 

Note that at initialization, $[\ul{\beta_L},\ol{\beta_L}]\cap \Gamma^\beta\neq \emptyset$ and $[\inf B,\ul{\beta_L}]\cap \Gamma^\beta=\emptyset$ (for the right side, we have $[\ul{\beta_U},\ol{\beta_U}]\cap \Gamma^\beta\neq \emptyset$ and $[\ol {\beta_U},\sup B]\cap \Gamma^\beta=\emptyset$). 
At each iteration of the algorithm, if $[\ul {\beta_L},\beta^*_L]\cap \Gamma^\beta \neq \emptyset$, then we set $\ol {\beta_L} = \beta^*_L$ and otherwise, we set $\ul {\beta_L}=\beta^*_L$. We see that after every iteration, we maintain that $[\ul {\beta_L},\ol {\beta_L}]\cap \Gamma^\beta\neq \emptyset$ and $[\inf B, \ul {\beta_L}] \cap \Gamma^\beta=\emptyset$  (for the right side, we have $[\ul{\beta_U},\ol{\beta_U}]\cap \Gamma^\beta\neq \emptyset$ and $[\ol {\beta_U},\sup B]\cap \Gamma^\beta=\emptyset$).

We see that at any point in the procedure, we have that $[\ul {\beta_L},\ol {\beta_U}]\supset \Gamma^\beta$, so the interval has coverage $(1-\alpha)$, inherited by the properties of $\Gamma^\beta$. 

In terms of the width of the interval, note that at any point in the procedure, $\inf \Gamma^\beta \in [\ul {\beta_L},\ol {\beta_L}]$ and $\sup \Gamma^\beta \in [\ul {\beta_U},\ol {\beta_U}]$. So, when the algorithm terminates, we have the following inequalities:
\[\inf \Gamma^\beta - \ul {\beta_L} \leq \ol {\beta_L} - \ul {\beta_L} <\mathrm{tol},\]
\[\ol {\beta_U} - \sup \Gamma^\beta \leq \ol {\beta_U} - \ol {\beta_U} <\mathrm{tol}.\]
This implies the following upper bound:
\begin{align*}
    \ol {\beta_U} - \ul {\beta_L} 
    &=\sup \Gamma^\beta - \inf \Gamma^\beta + (\ol {\beta_U} - \sup \Gamma^\beta + \inf \Gamma^\beta - \ul {\beta_L})\\
    &< \sup \Gamma^\beta - \inf \Gamma^\beta + 2\mathrm{tol},
\end{align*}
establishing that the interval produced by Algorithm 2 is at most $2\mathrm{tol}$ units wider than the smallest interval containing $\Gamma^\beta$.  
\end{proof}
}

\begin{proof}[Proof of Theorem 4.1.]
For any $\theta_0\in \Theta_0$, call $p(s;\theta_0) = \frac{1}{R+1}[\#\{T_{(i)}(\theta_0)\leq T_{\text{obs}}(\theta_0)\}]$, which due to the exchangeability of the $T_{(i)}(\theta_0)$ is a valid $p$-value for $H_0: \theta^*=\theta_0$, according to Lemma \ref{lem:domUnif}. Then we can write $p=\frac{ \sup_{\theta\in \Theta_0} \#\{T_{(i)}(\theta)\leq T_{\text{obs}}(\theta)\}}{R+1}$ as 
$p = \sup_{\theta\in \Theta_0} p(s;\theta).$
Then, since $p\geq p(s;\theta^*)$, where $\theta^*\in \Theta_0$ is the true parameter value, we have 
$P_{s\sim F_{\theta^*}}(p\leq \alpha)\leq P_{s\sim F_{\theta^*}}(p(s;\theta^*)\leq \alpha)\leq \alpha,$ 
and we see that $p$ is a valid $p$-value for $H_0$. 
\end{proof}
\begin{proof}[Proof of Proposition 4.2.]
    Note that since $T$ takes values in $[0,1]$, $\lfloor M\rfloor =\sup_{\theta\in \Theta_0}  \#\{T_i(\theta)\leq T_{\text{obs}}(\theta)\}$ unless $T_{\text{obs}}(\theta)=1$. When  $T_{\text{obs}}(\theta)=1$ we have $T_{\text{obs}}(\theta)\geq T_{i}(\theta)$ for all $i=1,2,\ldots, R$, giving $\lfloor M\rfloor =R+1$, whereas  $\#\{T_i(\theta)\leq T_{\text{obs}}(\theta)\}=R$. We conclude that $ \min\{\lfloor M\rfloor+1,R+1\}=\sup_{\theta\in \Theta_0}  \#\{T_i(\theta)\leq T_{\text{obs}}(\theta)\}+1=\sup_{\theta\in \Theta_0}  \#\{T_{(i)}(\theta)\leq T_{\text{obs}}(\theta)\}$.
We see that $p$, the output of $\texttt{pvalue}(\Theta_0,(u_i)_{i=1}^R,G,s,T)$ is equal to $\frac{ \sup_{\theta\in \Theta_0} \#\{T_{(i)}(\theta)\leq T_{\text{obs}}(\theta)\}}{R+1}$, as desired.
\end{proof}

\begin{prop}\label{prop:rate}
Let $\alpha \in (0,1)$. Then  $\sqrt{\log(1/(1-(1-\alpha)^{1/d}))}\sim \sqrt{\log(d)}$, 
as $d\rightarrow \infty$.
\end{prop}
\begin{proof}
Our goal is to show that 
$\lim_{d\rightarrow \infty} \frac{\sqrt{(-1)\log(1-(1-\alpha)^{1/d})}}{ \sqrt{\log(d)}}=1.$ 
First, we will evaluate the limit without the square root:
\begin{align*}
\lim_{d\rightarrow \infty} \frac{(-1)\log(1-(1-\alpha)^{1/d})}{\log(d)}
&\overset{\text{L'H}}=\lim_{d\rightarrow \infty} \frac{d(-1) \log(1-\alpha) (1-\alpha)^{1/d}}{d^2(1-(1-\alpha)^{1/d})}\\
&\overset{\phantom{\text{L'H}}}=\lim_{d\rightarrow \infty} \frac{(-1) \log(1-\alpha) (1-\alpha)^{1/d}(1/d)}{1-(1-\alpha)^{1/d}}\\
&\overset{\text{L'H}}{=} \lim_{d\rightarrow \infty} \frac{\log(1-\alpha)(1-\alpha)^{1/d} (\log(1-\alpha)+d)d^2}{d^3\log(1-\alpha)(1-\alpha)^{1/d}} \\
&\overset{\phantom{\text{L'H}}}=\lim_{d\rightarrow\infty} \frac{\log(1-\alpha)+d}{d}\\
&=1,
\end{align*}
where ``$\overset {\text{L'H}}=$'' indicates that we used L'H\^opital's rule. Finally, since $\sqrt{\cdot}$ is a continuous function at $1$, applying $\sqrt{\cdot}$ to both sides gives the desired result.
\end{proof}

\section{Additional simulation details and results}

\subsection{Exponential distribution}\label{s:exp}
{ In this section, we a privatized summary of exponentially distributed data, which is clamped in order to ensure finite sensitivity. We show that in this case we can exactly evaluate the DP summary's sampling distribution using the inversion of characteristic functions. We then compare the performance of exact confidence intervals at different clamping thresholds using the inversion technique, our simulation-based inference method, and the parametric bootstrap.}

Suppose that $x_1,\ldots, x_n \iid \mathrm{Exp}(\mu^*)$, where $\mu^*$ is the scale/mean parameter. We assume that we have some pre-determined threshold $c$, and the data will be clamped to lie in the interval $[0,c]$. Our privatized statistic is $s= \frac{1}{n} \sum_{i=1}^n [x_i]_0^c+\frac{c}{n\ep} N$, where $N$ is a noise adding mechanism; for this example, we will assume that $N\sim \mathrm{Laplace}(0,1)$, which ensures that $s$ satisfies $\ep$-DP. Note that there are two complications to understanding the sampling distribution of $s$: the clamping and the noise addition. 

We will show that for the exponential distribution, we can derive the characteristic function for $[x_i]_0^c$, which we will call $\phi_{\mu^*}$. Then the characteristic function for $s$ is 
\[\phi_{s,\mu^*}(t) = (\phi_{\mu^*}(t/n))^n\phi_N((c/(n\ep))t),\] where $\phi_N$ is the characteristic function for $N$. Once we have $\phi_{s}$, we can use inversion formulae to evaluate the cdf and pdf. This strategy was first proposed for DP statistical analysis in \citet{awan2023canonical} in the context of privatized difference of proportions tests. 
\begin{align*}
    \phi_{\mu^*}(t)&=\EE_{x\sim \mu^*} e^{it[x]_0^c}\\
    &=\int_0^c e^{itx} \frac{1}{\mu^*} e^{-x/\mu^*}\ dx + e^{itc}P_{x\sim \mu^*}( x\geq c)\\
    &=\int_0^c \frac{1}{\mu^*} e^{x(-1/\mu^*+it)}\ dx + e^{itc}e^{-c/\mu^*}\\
    &=\frac{e^{x(-1/\mu^*+it)}}{\mu^*(-1/\mu^*+it)} \Big|_0^c+e^{itc-c/\mu^*}\\
    &=\frac{1-e^{c(it-1/\mu^*)}}{1-it\mu^*} + e^{itc-c/\mu^*}.
\end{align*}

Note that while $[x]_0^c$ is not a continuous distribution, the Laplace noise makes $s$ continuous. So, Gil-Pelaez inversion tells us that we can evaluate the cdf $F_{s,\mu^*}$ as
\[F_{s,\mu^*}(x) = \frac 12 -\frac 1\pi \int_0^\infty \frac{\mathrm{Im}(e^{-itx} \phi_{s,\mu^*}(t))}{t} \ dt,\]
which can be solved by numerical integration, such as by \texttt{integrate} command in R. We can then use the cdf to determine a level $\alpha$ set by finding the $(1-\alpha)/2$ and $1-(1-\alpha)/2$ quantiles, for example by the bisection method (relies on the fact that $F_{s,\mu^*}(x)$ is monotone in $\mu^*$ for a fixed $s$ and $x$).  Based on the general description above, we then determine which $\mu^*$ values lie in these intervals. However, we can actually skip the creation of the level $\alpha$ sets by simply checking for each $\mu^*$ whether $F_{s,\mu^*}(s)\in [(1-\alpha)/2,1-(1-\alpha)/2]$. In this case, we only need to perform two bisection method searches to determine the lower and upper confidence limits for $\mu^*$. 

Note that by using characteristic functions, the computational complexity does not depend on $n$ at all, making this approach tractable for both small and large sample sizes.  While there can potentially be some numerical instability in the integration step, we found that it was not a significant problem for the settings considered in the following simulation.

In Table \ref{tab:CI_exp}, we see the result of a simulation study, using the confidence interval procedure described above for the exponential distribution. We used the parameters $n=100$, $\ep=1$, $\mu^*=10$, $\alpha=0.95$, and varied $c$. We include the empirical coverage and average confidence interval width over 1000 replicates. The coverage is very close to the provable $0.95$ coverage level. While one may expect the width to depend on the $c$ parameter, and we did indeed find larger widths with very small and very large $c$, even with extreme values of $c$, we still had an informative confidence interval, with the width of the smallest interval (from $c=20$) being 4.906, and the width of the largest interval (from $c=100$) being 7.153. { As the inversion method is a theoretical approach following the idea of Lemma 3.1, we compare it to our simulation-based inference method using repro with Mahalanobis depth on $s$, and see that the coverage and width of these two methods are similar. We also compare these approaches to parametric bootstrap (PB) simplified $t$ confidence interval (as in Section 5.2) where PB fails to provide enough coverage when $c=10$ and $c=20$.}

\begin{table}[t]
    \centering
    \begin{tabular}{ll|llll}
       &  & $c=10$&$c=20$&$c=50$&$c=100$ \\\hline
    &      Coverage&0.936 (0.008)&0.935 (0.008)&0.938 (0.008)&0.955 (0.007)\\
Inversion    &     Average Width&6.066 (0.041)&4.919 (0.023)&5.030 (0.014)&7.144 (0.011)\\\hline
    &     { Coverage} & { 0.944 (0.007)}&{ 0.942 (0.007)}&{ 0.949 (0.007)}&{ 0.955 (0.007)}\\
{ Repro Sample}    &     { Average Width}&{ 6.071 (0.044)}&{ 4.956 (0.025)}&{ 5.120 (0.019)}&{ 7.307 (0.024)}\\\hline
    &      Coverage& 0.005 (0.002)&0.785 (0.013)&0.947 (0.007)&0.957 (0.006)\\
 PB   &     Average Width&1.474 (0.001)&2.696 (0.004)&4.692 (0.011)&7.007 (0.014)
    \end{tabular}
    \caption{95\% confidence intervals for clamped exponential distribution with Laplace noise. $n=100$, $\ep=1$, $\mu^*=10$, number of replicates $1000$. Top row is based on the inversion of characteristic functions, { middle is based on the simulation-based inference of repro using Mahalanobis depth, and} bottom is from the parametric bootstrap (simplified-$t$ intervals).}
    \label{tab:CI_exp}
\end{table}

\subsection{More simulations with simple linear regression hypothesis testing}\label{s:LR_HT2}
We show more simulation results in Figure \ref{fig:Repro_LR_HT} and \ref{fig:PB_LR_HT}, where we vary the privacy parameter $\mu$. We can see that when $\mu\geq0.5$ for $\mu$-GDP and sample size $n=5000$, the type I error of the parametric bootstrap method \citep{alabi2022hypothesis} is higher than the significance level $\alpha=0.05$, while the repro sample method always controls the type I error well. Although the power of the two methods is comparable, the repro sample method follows our intuition that when $\beta_1$ is further away from 0, the power is larger, for all settings of sample size and privacy requirement, while the parametric bootstrap method does not.  
\begin{figure}[!htbp]
\centering
    \begin{minipage}{\textwidth}
        \centering
        \includegraphics[width=.99\linewidth]{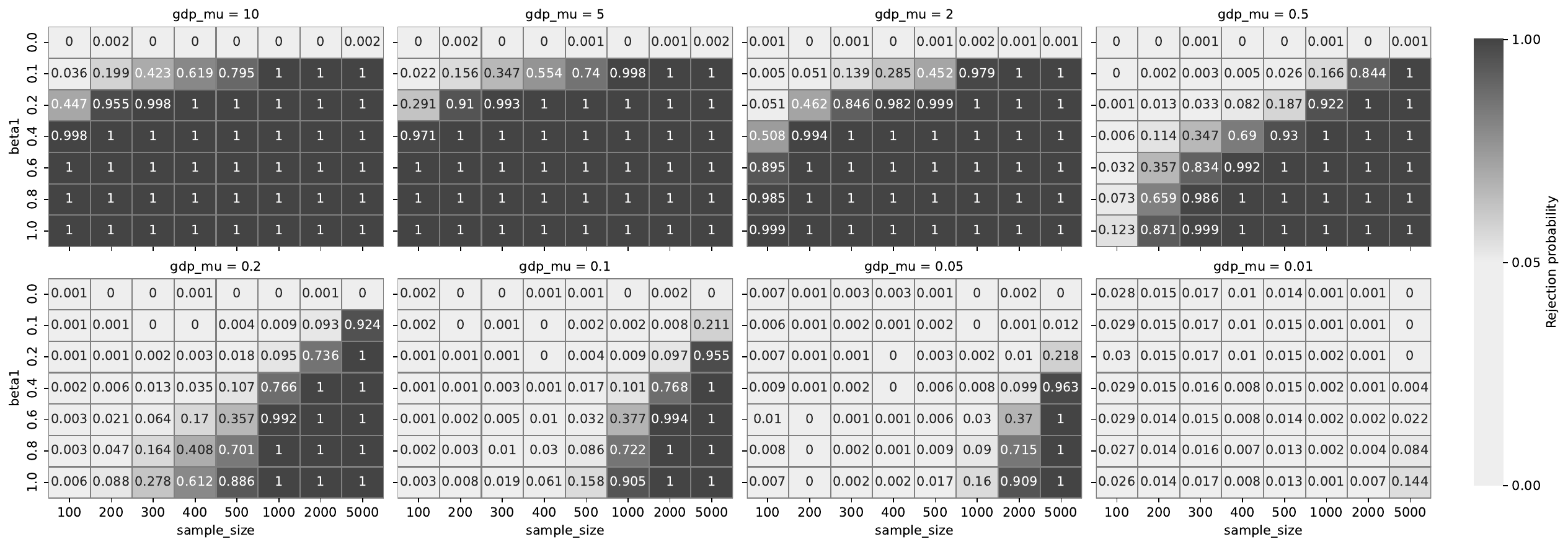}
        \caption{The rejection probability for hypothesis testing on $H_0:\beta_1^*=0$ and $H_1:\beta_1^*\neq 0$ in a linear regression model $Y=\beta_0^* +X\beta_1^* +  \epsilon$ with simulation-based repro sample method. We compare the results with different privacy settings of $\mu$-GDP: $\mu=10$, 5, 2, 0.5, 0.2, 0.1, 0.05, 0.01. (Figure 4 of Section 5.3 shows the result of $\mu=1$.) The significance level is 0.05, and the values in the table are calculated from 1000 replicates.}
        \label{fig:Repro_LR_HT}
    \end{minipage}
    \begin{minipage}{\textwidth}
        \centering
        \includegraphics[width=.99\linewidth]{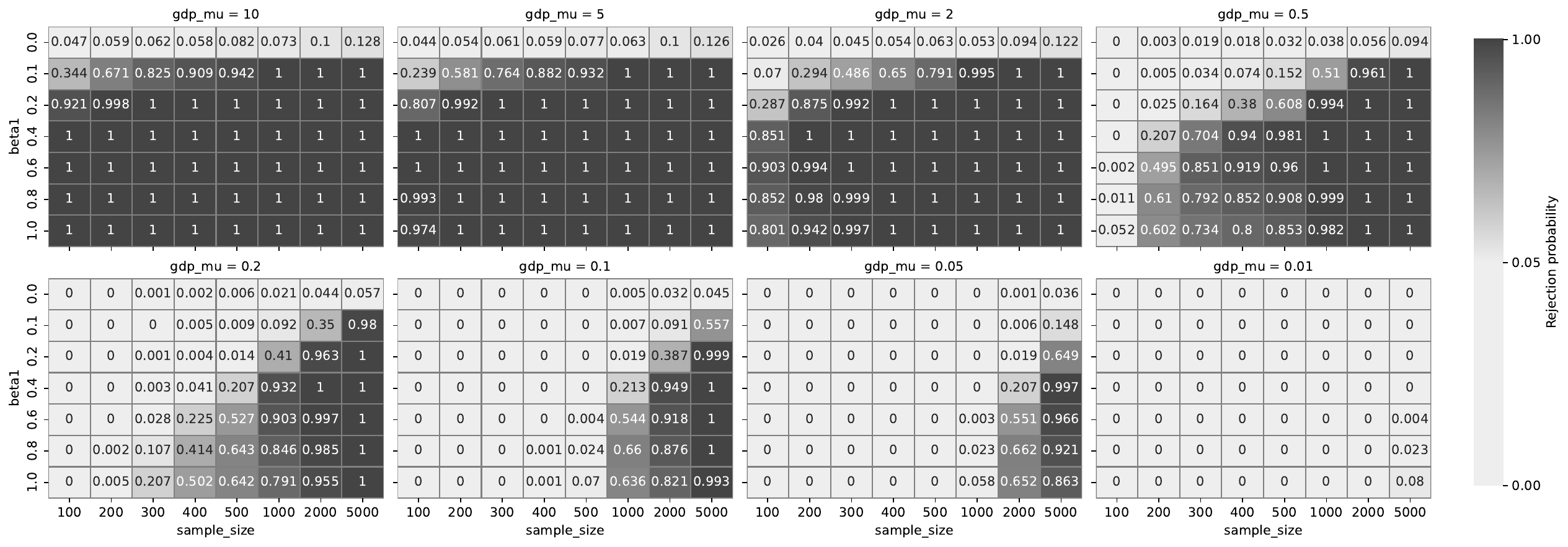}
        \caption{The rejection probability for hypothesis testing on $H_0:\beta_1^*=0$ and $H_1:\beta_1^*\neq 0$ in a linear regression model $Y=\beta_0^* +X\beta_1^* +  \epsilon$ with parametric bootstrap \citep{alabi2022hypothesis}. We compare the results with different privacy settings of $\mu$-GDP: $\mu=10$, 5, 2, 0.5, 0.2, 0.1, 0.05, 0.01. (Figure 4 of Section 5.3 shows the result of $\mu=1$.) The significance level is 0.05, and the values in the table are calculated from 1000 replicates.}
        \label{fig:PB_LR_HT}
    \end{minipage}
\end{figure}

\subsection{Details for logistic regression simulation}\label{s:logistic2}
We assume that the space of datasets is of the form $\mscr D = \mscr X^n$, and that the adjacency metric $d$ is the Hamming distance. For the logistic regression application in Section 5.4, we use techniques from \citet{awan2021structure}, including the sensitivity space, $K$-norm mechanisms, and objective perturbation. In order to be self-contained, we include the basic definitions and algorithms in this section. 

\begin{defn}[Sensitivity Space: \citealp{awan2021structure}]\label{AdjacentOutput}
Let $T: \mscr X^n\rightarrow \RR^m$ be any function. The \emph{sensitivity space} of $T$ is 
\[S_T = \l\{ u \in \RR^m \middle| \begin{array}{c}\exists X,X'\in \mscr X^n \text{ s.t } d(X,X')=1\\ \text{and }u=T(X)-T(X')\end{array}\r\}.\]
\end{defn}

A set $K\subset \RR^m$ is a \emph{norm ball} if $K$ is 
\begin{inparaenum}[1)]\item convex, \item bounded,  \item absorbing: $\forall u \in \RR^m$, $\exists c>0$ such that $u\in cK$, and \item
 symmetric about zero: if $u\in K$, then $-u\in K$. \end{inparaenum} 
If $K\subset \RR^m$ is a norm ball, then $\lVert u \rVert_K = \inf \{c\in \RR^{\geq 0}\mid u\in cK\}$ is a norm on $\RR^m$. 

\begin{algorithm}
\caption{Sampling  $\propto \exp(-c \lVert V\rVert_\infty)$   \citep{steinke2016between}}
\scriptsize
INPUT: $c$, and dimension $m$
\begin{algorithmic}[1]
\STATE Set $U_j \iid U(-1,1)$ for $j=1,\ldots, m$
\STATE Draw $r \sim \mathrm{Gamma}(\al = m+1,\beta = c)$
\STATE Set $V = r\cdot (U_1,\ldots, U_m)^\top$
\end{algorithmic}
OUTPUT: $V$
\label{alg:SampleLinfty}
\end{algorithm}

The sensitivity of a statistic $T$ is the largest amount that $T$ changes when one entry of $T$ is modified. Geometrically, the sensitivity of $T$ is the largest radius of $S_T$ measured by the norm of interest. For a norm ball $K \subset \RR^m$, the \emph{$K$-norm sensitivity} of $T$ is 
\[\Delta_K(T) = \sup_{d(X,X')=1} \lVert T(X) - T(X')\rVert_K = \sup_{u \in S_T} \lVert u \rVert_K .\]

\begin{algorithm}
\caption{Sampling  $K$-Norm Mechanism with Rejection Sampling \citep{awan2021structure}}
\label{RejectionAlgorithm}
\scriptsize
INPUT: $\ep$, statistic $T(X)$, $\ell_\infty$ sensitivity $\Delta_\infty(T)$, $K$-norm sensitivity $\Delta_K(T)$
\begin{algorithmic}[1]
\STATE Set $m = \mathrm{length} (T(X))$.
\STATE Draw $r \sim \mathrm{Gamma}(\al = m+1, \beta = \ep/\Delta_K(T))$
\STATE Draw $U_j \iid \mathrm{Uniform}(-\Delta_\infty(T), \Delta_\infty(T))$ for $j=1,\ldots, m$
\STATE Set $U = (U_1,\ldots, U_m)^\top$
\STATE If $U\in K$, set $N_K=U$, else go to 3)
\STATE Release $T(X) + r\cdot N_K$.
\end{algorithmic}
\label{alg:sampleK}
\end{algorithm}

\begin{algorithm}
\caption{$\ell_\infty$ Objective Perturbation \citep{awan2021structure}}
\scriptsize
INPUT: $X\in \mscr{ X}^n$, $\ep>0$, a convex set $\Theta \subset \RR^m$, a convex function $r: \Theta\rightarrow \RR$, a convex loss  $\hat{\mscr  L}(\ta; X) =\frac1n \sum_{i=1}^n \ell(\ta;x_i)$ defined on $\Theta$ such that $\nabla^2 \ell(\ta;x)$ is continuous in $\ta$ and $x$, $\De>0$ such that $\sup_{x,x'\in \mscr X} \sup_{\ta\in \Ta}\lVert \nabla \ell(\ta;x) - \nabla\ell(\ta;x')\rVert_\infty\leq \De$,  $\la>0$  is an upper bound on the eigenvalues of $\nabla^2\ell(\ta;x)$ for all $\ta\in \Theta$ and $x\in \mscr X$, and $q\in (0,1)$.
\begin{algorithmic}[1]
  \setlength\itemsep{0em}
  \STATE Set $\ga = \frac{\la}{\exp({\ep(1-q)})-1}$
\STATE Draw $V\in \RR^m$ from the density $f(V;\ep, \De)\propto \exp(-\frac{\ep q}{\De}\lVert V\rVert_\infty)$ using Algorithm \ref{alg:SampleLinfty}
\STATE Compute $\ta_{DP} = \arg\min_{\ta\in \Theta} \hat{\mscr L}(\ta;X) +\frac1n r(\theta)+ \frac{\ga}{2n} \ta^\top \ta + \frac{V^\top \ta}{n}$
\end{algorithmic}
OUTPUT: $\ta_{DP}$
\label{alg:ExtendedObjPert}
\end{algorithm}

The privacy mechanism consists of two components, each of dimension 2: to encode information about the parameters $a$ and $b$, we release $(\sum_{i=1}^n z_i,\sum_{i=1}^n z_i^2)+N_K$, where $N_K$ is from a $K$-norm mechanism that discuss in the following paragraphs; to learn $\beta_0$ and $\beta_1$, we run Algorithm \ref{alg:ExtendedObjPert}. In \citet[Section 4.1]{awan2021structure}, it was calculated that $\Delta_\infty=2$ and $\lambda=m/4$, where $m$ is the number of regression coefficients; in \citet[Section 4.2]{awan2021structure}, it was numerically found that $q=0.85$ optimized the performance of objective perturbation. \citet[Theorem 4.1]{awan2021structure} showed that the output of Algorithm \ref{alg:ExtendedObjPert} satisfies $\ep$-DP. 

\citet{awan2021structure} showed that the optimal $K$-norm mechanism (in terms of quantities such as stochastic tightness, minimum entropy, and minimum conditional variance) of a statistic $T(X)$ uses the norm ball, which is the convex hull of the sensitivity space. For the statistic $T(X) = (\sum_{i=1}^n z_i,\sum_{i=1}^n z_i^2)$, where $z_i \in [0,1]$, we can derive the sensitivity space as 
\begin{figure}
    \centering
    \includegraphics[width=.4\linewidth]{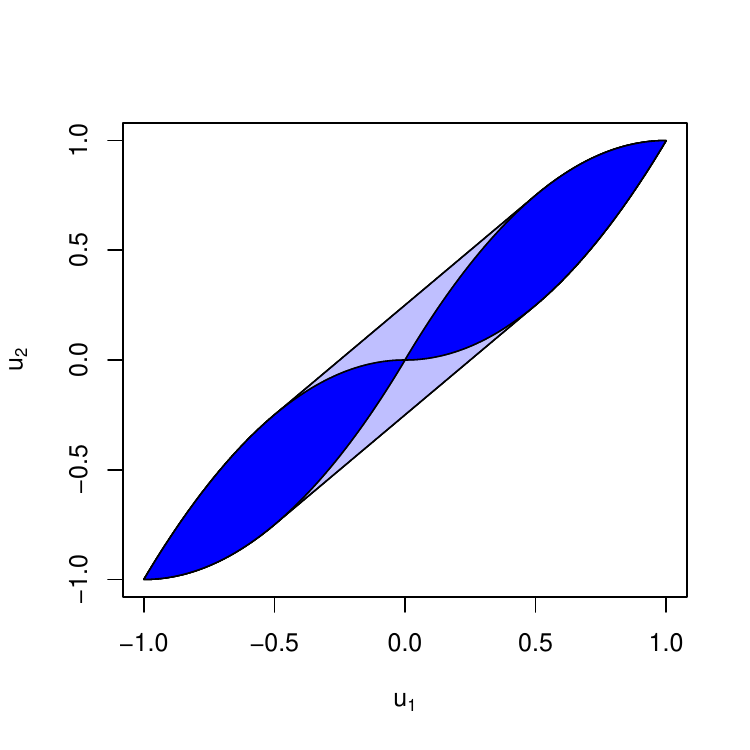}
    \caption{In dark blue is the sensitivity space derived in Equation \eqref{eq:sensitivity}. In light blue and dark blue is the convex hull of the sensitivity space, derived in Equation \eqref{eq:hull}. }
    \label{fig:sensitivityHull}
\end{figure}
\begin{equation}\label{eq:sensitivity}
    S_T = \left\{(u_1,u_2) \middle| \begin{array}{c}
u_1^2\leq u_2\leq 1-(u_1-1)^2\\
\text{or }\quad (1+u_1)^2-1\leq u_2\leq u_1^2
\end{array}\right\},
\end{equation}
and the convex hull can be expressed as 
\begin{equation}\label{eq:hull}
\mathrm{Hull}(S_T) = \left\{ (u_1,u_2) \middle| 
\begin{array}{cc}
(u_1+1)^2-1\leq u_2\leq -u_1^2,& u_1\leq -1/2\\
u_1-1/4\leq u_2\leq u_1+1/4,& -1/2<u_1\leq 1/2\\
u_1^2\leq u_2\leq 1-(u_1-1)^2,&1/2\leq u_1
\end{array}\right\}.\end{equation}
See Figure \ref{fig:sensitivityHull} for a visualization of $S_T$ and $\mathrm{Hull}(S_T)$. 
We use $K=\mathrm{Hull}(S_T)$ as the norm ball in Algorithm \ref{alg:sampleK} with $\Delta_K=1$ and $\Delta_\infty=1$,  which \citet{hardt2010geometry} showed satisfies $\ep$-DP. 

In the simulation generating the results in Figure 6 of Section 5.4, we let $a^*=b^*=0.5$, $\beta_0^*=0.5$, $\beta_1^*=2$, $R=200$, $\alpha=0.05$, $n=100,$ 200, 500, 1000, 2000, and $\ep=0.1$, 0.3, 1, 3, 10 in $\ep$-DP. For the repro sample method, the private statistics we release and use are $s=(\tilde{\beta}_0, \tilde{\beta}_1, \tilde T(X))$ where $(\tilde{\beta}_0$, $\tilde{\beta}_1)$ are the output of Algorithm \ref{alg:ExtendedObjPert} with $0.9\ep$-DP, $q=0.85$, $\lambda=1/2$, and $\tilde T(X)$ is the output of Algorithm \ref{alg:sampleK},  with private budget $0.1\ep$. When using Algorithm \ref{alg:SampleLinfty} and \ref{alg:sampleK}, we let $m=2$ as it is the dimension of the output. For the DP-CI-ERM \citep{wang2019differentially}, we use their Algorithm 1 (Objective Perturbation) with $\ep/2$-DP, $t=1/4$, $c=t/(2n)/(\mathrm{e}^{(1-q)\ep/2}-1)$, $q=0.85$, and we use their Algorithm 2 in \citep{wang2019differentially} to obtain the private estimations of the covariance matrix and the Hessian matrix, both with $\ep/4$-DP and the same $c$ as in the Objective Perturbation; in their Algorithm 5, we set $m=10,000$. For both the repro sample method and DP-CI-ERM, we compute the confidence interval within the range $[-10,10]$.

\subsection{Private Bernoullis with unknown \texorpdfstring{$n$}{n}}\label{s:bernoulli}

For this example, we derive a confidence interval for a population proportion $p^*$ from i.i.d. Bernoulli data, under unbounded-DP. 
\citet{awan2018differentially}, \citet{awan2020differentially}, and \citet{awan2023canonical} derived optimal private hypothesis tests and confidence intervals for Bernoulli data, under $(\ep,\de)$-DP as well as general $f$-DP, but assumed that the sample size $n$ was known as they worked in the \emph{bounded}-DP setting. In the framework of \emph{bounded}-DP, it is assumed that $n$ is public knowledge, and two databases are neighboring if they differ in one entry. On the other hand, in \emph{unbounded}-DP, it is assumed that even $n$ itself may be sensitive; in this framework, two databases of differing sizes are neighbors if one can be obtained from the other by adding or removing an entry (individual). DP inference has largely focused on the bounded-DP framework, as it is challenging to perform valid inference when $n$ is unknown. While some papers argue that a noisy estimate of $n$ can be plugged in to obtain asymptotically accurate inference (e.g., \citet{smith2011privacy}), ensuring valid finite sample inference is more challenging. 

 Call $x_1,\ldots, x_{n^*}\iid\mathrm{Bern}(p^*)$, where both $p^*$ and $n^*$ are unknown; so our full parameter is $\theta^* =(p^*,n^*)$. For privacy, we observe both a noisy count of 1's and a noisy count of 0's: $s^{(1)} = \sum_{i=1}^{n^*} x_i + N_1$, $s^{(2)} = n^*-\sum_{i=1}^{n^*} x_i + N_2$, where $N_1,N_2\iid N(0,1/\ep^2)$, which satisfies $\ep$-GDP as the $\ell_2$ sensitivity of $(\sum_{i=1}^{n^*}x_i,n^*-\sum_{i=1}^{n^*} x_i)$ is 1. Using the notation of Section 3, let $u^{(1)},u^{(2)},u^{(3)}\iid U(0,1)$, and set $G(\theta^*,u) = (F^{-1}(u^{(1)})+\Phi^{-1}(u^{(2)})/\ep,\ n^*-F^{-1}(u^{(1)})+\Phi^{-1}(u^{(3)})/\ep)$, where $F^{-1}$ is the quantile function of $\mathrm{Binom}(n^*,p^*)$ and $\Phi(\cdot)$ is the CDF of standard normal distribution. We see that $G(\theta^*,u)$ is equally distributed as $(s^{(1)},s^{(2)})$ described above. 

\begin{table}[t]
    \centering
    \begin{tabular}{c|cc}
    &Coverage&Width\\\hline
       Mahalanobis Depth  & 0.981 (0.004)& 0.198 (0.0006)\\
        Approximate Pivot &0.949 (0.007) &0.164 (0.0005)
    \end{tabular}
    \caption{$95\%$ confidence intervals for private Bernoullis with unknown $n$. The first row uses Mahalanobis depth on $(s^{(1)}, s^{(2)})$, and the second row uses Mahalanobis depth of the test statistics from Equation \eqref{eq:Tbinom}. For both intervals, an initial ($1-10^{-4}$)-CI for $n$ is used to reduce the nuisance parameter search. Parameters for the simulation are $n^*=100$, $p^*=0.2$, $\epsilon=1$, $R=200$, and the results were averaged over 1000 replicates. }
    \label{tab:Binom}
\end{table}

First, we consider the $95\%$ confidence interval built by using Mahalanobis depth, whose width and coverage are shown in Table \ref{tab:Binom}. We see that the coverage is 0.98,  indicating over-coverage. Alternatively, we consider the test statistic
\begin{equation}\label{eq:Tbinom}
    T_\theta(s) = \frac{s^{(1)}-\hat np}{\sqrt{\hat np(1-p)+(p^2+(1-p)^2)\ep^{-2}}},
\end{equation}
where $\hat n = \max\{s^{(1)}+s^{(2)},1\}$, which we can see is asymptotically distributed as $N(0,1)$, as $n^*\rightarrow \infty$.

We then apply our simulation-based methodology, but use $T_\theta(s),T_\theta(s_1(\theta)),\ldots, T_\theta(s_R(\theta))$. Note that we need to search over the nuisance parameter $n$ as well. We propose to first get a very conservative confidence interval for $n^*$ based on $(s^{(1)},s^{(2)})$ (say with coverage $1-10^{-4}$), which is easily done since $s^{(1)}+s^{(2)}\overset d= n^*+N(0,2/\ep^2)$, and $\ep$ is a known value. Then  we set the nominal coverage level of the repro interval for $\theta$ as $1-\alpha+10^{-4}$, and only search over the values of $n$ in its preliminary confidence interval. The idea of using a preliminary confidence set to limit the search space was also proposed in \citet{xie2022repro}.

In Table \ref{tab:Binom}, we implement the repro interval using both the Mahalanobis depth, with either $(s^{(1)},s^{(2)})$ or the approximate pivot $T_\theta(s)$ given in \eqref{eq:Tbinom}  using the preliminary search interval for $n^*$. In the simulation described in Table \ref{tab:Binom}, we see that using the approximate pivot $T_\theta(s)$ has better-calibrated coverage, and reduced width at 0.163 compared to 0.197 compared to the interval generated directly from $(s^{(1)},s^{(2)})$.

\begin{remark}
This example highlights one of the benefits of using Gaussian noise for privacy, rather than other sorts of noise distributions. Because many statistical quantities are approximately Normal, using Gaussian noise allows us to more easily construct approximate pivots. While many privacy mechanisms result in asymptotically normal sampling distributions (e.g., \citealp{smith2011privacy,wang2018statistical,awan2019benefits}), if the privacy noise is not itself Gaussian, then it has been found that such asymptotic approximations do not work well in finite samples \citep{wang2018statistical}. 
\end{remark}

{ 

\subsection{Private Mann–Whitney test}\label{s:dp_Mann–Whitney}

We use our simulation-based methodology to build a DP Mann–Whitney test with finite-sample type I error guarantees. \citet{couch2019differentially} built the first DP Mann–Whitney test where they used a parametric bootstrap method to determine the rejection region. We apply our simulation-based inference methodology using the same test statistic, and compare the two methods through a simulation study. We show that their test fails to control the type I error, while our method always has the type I error below the significance level. Furthermore, we compare against another approach for finite-sample DP hypothesis testing proposed by \citet{kazan2023test}, and we show that our results have better or comparable power to their results. 

We first review the sensitivity analysis of the non-private test statistic by \citet{couch2019differentially}. Let $D=(x_1,\ldots,x_n)$, and let $r_i$ be the rank of $x_i$ in $D$ for $i=1,\ldots,n$. Assume that $D$ contains two groups, $D_1=(x_1,\ldots, x_{n_1})$ and $D_2=(x_{n_1+1},\ldots, x_{n})$, where $n$ is fixed, $n_1$ is not publicly known, and $x_i$ are independently sampled from continuous distributions. We want to test whether the elements in $D_1$ and $D_2$ are from the same distribution. Let $m=\min(n_1, n-n_1)$, and $U_1=\sum_{i=1}^{n_1} r_i - \frac{n_1(n_1+1)}{2}$. The non-private Mann–Whitney test statistic is $U=\min(U_1,~ (n-n_1)n_1 - U_1)$. \citet{couch2019differentially} proved that the local sensitivity of $U$ is $(n-m)$ when the neighboring datasets can have one individual different in both value and group. By the definition of $m$, its global sensitivity is 1. As $n-m \leq n$ for any unknown $m$, we can derive from the local sensitivity result that the global sensitivity of $U$ is $n$. 

\citet{couch2019differentially} proposed an $(\ep, \delta)$-DP estimate of $(m, U)$ based on the local sensitivity of $U$ and global sensitivity of $m$. As the framework by \citet{kazan2023test} satisfies $\ep$-DP, to fairly compare \citep{couch2019differentially} with \citep{kazan2023test}, modify the approach of \citet{couch2019differentially} to satisfy $\ep$-DP by using the global sensitivity of $U$ and $m$ in the following way: we define $\tilde{m}=m+\frac{1}{\ep_m}u^{(1)}_{DP}$ and $\tilde{U}=U+\frac{n}{\ep_U}u^{(2)}_DP$ where $\ep_m+\ep_U = \ep$ and $u^{(i)}_{DP}\iid \mathrm{Laplace}(0,1)$, then, $s=(\tilde{m}, \tilde{U})$ satisfies $\ep$-DP. We can also replace the above Laplace mechanism with Gaussian mechanism such that $(\tilde{m}, \tilde{U})$ satisfies $\ep$-GDP. After obtaining the private estimate, \citet{couch2019differentially} used parametric bootstrap to obtain the $p$-value, and they mentioned that there was no theoretical guarantee but only empirical verification for the type I error control. 

In contrast to \citep{couch2019differentially}, \citet{kazan2023test} proposed the test of tests which has theoretically proven control of the type I error. In this method, \citet{kazan2023test} used subsample-and-aggregate \citep{nissim2007smooth} to transform a valid non-private test into a valid private test: They split the sensitive dataset into $k$ non-overlapping subsets, conducted the valid non-private hypothesis test with significance level $\alpha_0$ on each split, and built an $\ep$-DP test statistic by privatizing the number of rejections of these tests. As the number of rejections follows a binomial distribution $\mathrm{Binom}(k, \alpha_0)$ under the null hypothesis, \citet{kazan2023test} set the optimal private binomial test \citep{awan2018differentially} as their privatized test (under $\ep$-DP) to guarantee the type I error control. To transform their test to a $\ep$-GDP test, we simply change the Tulap distribution to a Gaussian distribution \citep{awan2023canonical}. Although \citep{kazan2023test} can be used in general test settings such as the Mann–Whitney test, its power is often much lower than \citep{couch2019differentially} even when their type I errors are similar, as shown by our simulation results in the remaining of this section. 

After introducing the approaches of \citep{couch2019differentially} and \citep{kazan2023test}, we demonstrate how to use simulation-based inference to analyze the DP summaries from \citet{couch2019differentially}.
Under the null hypothesis, 
\begin{center}
    $H_0:$ $D_1$ and $D_2$ are from the same distribution $F$,
\end{center}
the distribution of $U$ does not depend on $F$ but only depends on $m$. Therefore, we can simulate $U$ with $x_i\iid F=\mathrm{Unif}(0,1)$. For the repro method, we let $\hat{m} = \min(\max(\tilde{m}, 1), \lfloor \frac{n}{2} \rfloor)$ and compute the pivot test statistic, $T(s)=\frac{\tilde{U} - \frac{(n-\tilde{m})\tilde{m}}{2} - \frac{1}{\ep_m^2} }{\sqrt{\frac{\hat{m}(n - \hat{m})(n+1)}{12} +\frac{2n^2}{\ep_u^2} + \frac{(n-2\hat{m})^2}{2\ep_m^2} + \frac{1}{2\ep_m^4} }}$. We reject the null hypothesis if the $p$-value computed by our Theorem 4.1 is smaller than or equal to the significance level $\alpha$ where $\Theta_0$ includes all choices of $m$.

Given $\ep=1$, $\alpha=0.05$, $n=100$, we compare our method with \citep{couch2019differentially} and \citep{kazan2023test}
for the settings of $m=20,~ 30,~ 50$ under both $\ep$-DP and $\ep$-GDP. We manually choose the hyperparameter for \citep{kazan2023test} to obtain its best performance in these simulations. For repro and the $\ep$-DP version of \citep{couch2019differentially}, we try $\ep_m=0.2,~0.3,~0.4$. We use $R=1000$ in repro. For calculating the type I error, we use $x_i\iid \mathrm{Unif}(0,1)$, $i=1,\ldots, n$. For calculating the power, we use $x_i\iid \mathrm{Unif}(0,1)$, $i=1,\ldots, m$, and $x_j\iid \mathrm{Beta}(2,5)$, $j=m+1,\ldots, n$. The results are shown in Table \ref{tab:mann_whitney_epDP} for $\ep$-DP and Table \ref{tab:mann_whitney} for $\ep$-GDP. We can see that \citep{couch2019differentially} fails to control the type I error when $m=20$ while having the largest power in all settings. Our method controls the type I error well, similar to \citep{kazan2023test}, while improving their power from $0.324$ and $0.342$ to $0.380$ and $0.606$ when $m=30$ and $m=50$ respectively under $\ep$-DP, and improving uniformly and more significantly under $\ep$-GDP. The increased performance under GDP is likely due to the fact that GDP has a tighter composition result than $\ep$-DP, and both the repro and \citet{couch2019differentially} method rely on composition, whereas \citet{kazan2023test} does not. 

In the end of this section, we show that there is a tradeoff between the performance of the repro sample method and the computation cost.
When using the repro method for DP Mann–Whitney test, as we need to compute the $p$-value using $\sup_{\theta\in\Theta_0}$ where $\Theta_0$ includes all choices of $m$, it is necessary to search over $m$. In Table \ref{tab:mann_whitney_gdp_compareR}, we show that increasing $R$ gives us a better power when searching over $m$, although it requires longer computation time due to more extensive simulations. Intuitively, larger $R$ enables more accurate simulation of the test statistic distribution to improve the power, but this improvement may not be significant when $R$ is large enough. However, in Table \ref{tab:mann_whitney_gdp_compareR}, when $m$ is unknown, compared to the case that $m$ is known and we do not need to search over $m$, the power and type I error still increase for larger $R$ when $R$ is already large. Note that the type I errors are  under control for all $R$, and these results only indicate that the repro method is less conservative for larger $R$. This is because when we search over $m$, we may reject the null hypothesis for the true $m$ but accept the null for another $m$, which means we finally reject less cases under the null, i.e., smaller $R$ increases the conservativeness of repro.


\begin{table}[t]
    \centering
    \begin{tabular}{l|ccc|ccc}
    &  \multicolumn{3}{c|}{Type I error} & \multicolumn{3}{c}{Power} \\
    & $m=20$ & $m=30$ & $m=50$ & $m=20$ & $m=30$ & $m=50$ \\\hline
    Repro Sample ($\ep_m=0.2$) & 0.055  & 0.040 & 0.045  & 0.205 &  0.363 & 0.606 \\
    Repro Sample ($\ep_m=0.3$) & 0.048 & 0.037 & 0.043 & 0.239 & 0.380 & 0.547 \\  
    Repro Sample ($\ep_m=0.4$) & 0.044  & 0.038  & 0.041  & 0.219 & 0.320 & 0.446 \\\hline  
    \citep{couch2019differentially} ($\ep_m=0.2$) & 0.112  & 0.065  & 0.051  & 0.437 & 0.522 & 0.641 \\
    \citep{couch2019differentially} ($\ep_m=0.3$) & 0.083 & 0.056 & 0.052 & 0.381 & 0.474 &  0.587 \\ 
    \citep{couch2019differentially} ($\ep_m=0.4$) & 0.069  & 0.053  & 0.055  & 0.325 & 0.388 & 0.510  \\\hline 
    \citep{kazan2023test} ($\ep=1$) & 0.046 &  0.051  & 0.047  & 0.259  & 0.324 & 0.342 \\
    \end{tabular}
    \caption{{ Privatized Mann–Whitney test under $\ep$-DP. Parameters for the simulation are $n=100$, $\ep=1$, $\alpha=0.05$. For Repro Sample and \citep{couch2019differentially}, the privacy budget $\ep$ is split into $\ep_m$ and $\ep_U=\ep-\ep_m$ for privatizing $m=\min(n_1,n_2)$ and $U$, the non-private Mann–Whitney test statistic. For calculating the type I error, we use $x_i\iid \mathrm{Unif}(0,1)$, $i=1,\ldots, n$; For calculating the power, we use $x_i\iid \mathrm{Unif}(0,1)$, $i=1,\ldots, m$, and $x_j\iid \mathrm{Beta}(2,5)$, $j=m+1,\ldots, n$. The results are over 1000 replicates.}}
    \label{tab:mann_whitney_epDP}
\end{table}

\begin{table}[t]
    \centering
    \begin{tabular}{l|ccc|ccc}
    &  \multicolumn{3}{c|}{Type I error} & \multicolumn{3}{c}{Power} \\
    & $m=20$ & $m=30$ & $m=50$ & $m=20$ & $m=30$ & $m=50$ \\\hline
    Repro Sample ($\ep_m^2=0.1$) & 0.041  & 0.039 & 0.054  & 0.456 &  0.642 & 0.822 \\
    Repro Sample ($\ep_m^2=0.2$) & 0.039 & 0.037 & 0.056 & 0.477 & 0.657 & 0.812 \\  
    Repro Sample ($\ep_m^2=0.5$) & 0.038  & 0.037  & 0.051  & 0.427 & 0.581 & 0.727 \\\hline  
    \citep{couch2019differentially} ($\ep_m^2=0.1$) & 0.097  & 0.069  & 0.051  & 0.609 & 0.727 & 0.827 \\
    \citep{couch2019differentially} ($\ep_m^2=0.2$) & 0.083  & 0.057  & 0.051  & 0.597 & 0.721 & 0.813 \\ 
    \citep{couch2019differentially} ($\ep_m^2=0.5$) & 0.071  & 0.051  & 0.053  & 0.520 & 0.638 & 0.728  \\\hline 
    \citep{kazan2023test} ($\ep=1$) & 0.042 &  0.043  & 0.035  & 0.351  & 0.411 & 0.475 \\
    \end{tabular}
    \caption{{ Privatized Mann–Whitney test under $\ep$-GDP. Parameters for the simulation are $n=100$, $\ep=1$, $\alpha=0.05$. For Repro Sample and \citep{couch2019differentially}, the privacy budget $\ep$ is split into $\ep_m$ and $\ep_U=\sqrt{\ep^2-\ep_m^2}$ for privatizing $m=\min(n_1,n_2)$ and $U$, the non-private Mann–Whitney test statistic. For calculating the type I error, we use $x_i\iid \mathrm{Unif}(0,1)$, $i=1,\ldots, n$; For calculating the power, we use $x_i\iid \mathrm{Unif}(0,1)$, $i=1,\ldots, m$, and $x_j\iid \mathrm{Beta}(2,5)$, $j=m+1,\ldots, n$. The results are over 1000 replicates.}}
    \label{tab:mann_whitney}
\end{table}

\begin{table}[t]
    \centering
    \begin{tabular}{l|ccc|ccc}
    &  \multicolumn{3}{c|}{Search $m$ ($m$ unknown)} & \multicolumn{3}{c}{Fix $m$ ($m$ known)} \\
    &  Type I error & Power & Time (s) &  Type I error & Power & Time (s) \\\hline
    $R=100$ & 0.023 & 0.397 &   8.227 & 0.052 & 0.519 & 3.195 \\
    $R=200$ & 0.030 & 0.441 &   12.849 & 0.054 & 0.526 & 2.270\\
    $R=500$ & 0.035 & 0.464 &   31.645 & 0.052 & 0.530 & 3.645\\
    $R=1000$ & 0.039 & 0.477 &   63.392 & 0.049 & 0.529 & 4.511\\
    $R=2000$ & 0.042 & 0.488 &   128.635 & 0.050 & 0.527 & 8.219\\
    $R=5000$ & 0.042 & 0.494 &   321.166 & 0.052 & 0.532 & 21.554 \\
    \end{tabular}
    \caption{{ Comparison of different $R$ in the repro sample method for the privatized Mann–Whitney test under $\ep$-GDP. Parameters for the simulation are $n=100$, $\ep=1$, $\alpha=0.05$, $m=20$, $\ep_m^2=0.2$. The privacy budget $\ep$ is split into $\ep_m$ and $\ep_U=\sqrt{\ep^2-\ep_m^2}$ for privatizing $m=\min(n_1,n_2)$ and $U$, the non-private Mann–Whitney test statistic.
    Using the repro method requires the search of $m$, and we compare it to the case that $m=20$ is given as prior knowledge.
    For calculating the type I error, we use $x_i\iid \mathrm{Unif}(0,1)$, $i=1,\ldots, n$; For calculating the power, we use $x_i\iid \mathrm{Unif}(0,1)$, $i=1,\ldots, m$, and $x_j\iid \mathrm{Beta}(2,5)$, $j=m+1,\ldots, n$. The results are over 1000 replicates.}}
    \label{tab:mann_whitney_gdp_compareR}
\end{table}
}


\end{document}